\documentclass[9pt]{article}
\usepackage[latin1]{inputenc}
\usepackage{amscd}
\usepackage{amsfonts}
\usepackage{amsgen}
\usepackage{amsmath}
\usepackage{mathabx}
\usepackage{amssymb}
\usepackage{amstext}
\usepackage{mathtools}
\usepackage{amsthm}
\usepackage{graphicx}
\usepackage{tikz-cd}
\usepackage{indentfirst}
\usepackage{latexsym}
\usepackage{makeidx}
\usepackage{mathrsfs}
\usepackage{stackrel}
\usepackage{epsfig}
\usepackage{wrapfig}
\usepackage[all,knot,arc,import,poly]{xy}

\newtheorem{theorem}{Theorem}

\newtheorem{definition}[theorem]{Definition}
\newtheorem{example}[theorem]{Example}

\newtheorem{lemma}[theorem]{Lemma}

\newtheorem{Df}{Definition}[section]
\newtheorem{Teo}[Df]{Theorem}
\newtheorem{Prop}[Df]{Proposition}
\newtheorem{Lem}[Df]{Lemma}
\newtheorem{Ex}[Df]{Example}
\newtheorem{Obs}[Df]{Remark}

\newtheorem{Que}[Df]{Question}

\newtheorem{fact}[Df]{Fact}
\newcommand{\n}{\noindent}
\newcommand{\Dem}{\n{\bf Proof:\;\;}}

\newcommand{\bc}{\begin{center}}
\newcommand{\ec}{\end{center}}

\newcommand{\N}{\mathbb{N}}

\newcommand{\vtres}{\vspace*{0.3cm}}

\newcommand{\cL}{\mbox{$\cal L$}}

\newcommand{\Lf}{\cL_f}

\def \beq { \begin{equation} }
\def \eeq { \end{equation} }

\def \rest {\restriction}

\newcommand{\baf}{\begin{Afi}\nf{\sl . }}
\newcommand{\eaf}{\end{Afi}}

\newcommand{\bdf}{\begin{Df}\nf{\bf .}}
\newcommand{\edf}{\end{Df}}
\newcommand{\bte}{\begin{Th}\nf{\bf .}}
\newcommand{\ete}{\end{Th}}  
\newcommand{\bco}{\begin{Co}\nf{\bf .}}
\newcommand{\eco}{\end{Co}}
\newcommand{\ble}{\begin{Le}\nf{\bf .}}
\newcommand{\ele}{\end{Le}}
\newcommand{\bpr}{\begin{Pro}\nf{\bf .}}
\newcommand{\epr}{\end{Pro}}
\newcommand{\bex}{\begin{Exa}\nf{\bf .} \rm}
\newcommand{\eex}{\end{Exa}}
\newcommand{\bre}{\begin{Rem}\nf{\bf .} \rm}
\newcommand{\ere}{\end{Rem}}
\newcommand{\bfa}{\begin{Fa}\nf{\bf .} \sl}
\newcommand{\efa}{\end{Fa}}
\newcommand{\bqt}{\begin{Que}\nf{\bf .}}
\newcommand{\eqt}{\end{Que}}

\newcommand{\bct}{\begin{Ct}\nf{\bf .} \rm}
\newcommand{\ect}{\end{Ct}}
\newcommand{\bxa}{\begin{Exa}\nf{\bf .} \rm}
\newcommand{\exa}{\end{Exa}}

\newcommand{\sub}{\subseteq}

\newcommand{\pres}{\mathbb{P}res}
\newcommand{\sig}{\mathbb{S}ig}
\newcommand{\cat}{\mathbb{C}at}
\newcommand{\CAT}{\mathbf{CAT}}
\newcommand{\Room}{\mathbb{R}oom}

\newcommand{\set}{\mathbb{S}et}

\newcommand{\inm}{\mathbf{Ins}_{mor}}
\newcommand{\inc}{\mathbf{Ins}_{co}}
\newcommand{\pim}{\pi\mathbf{Ins}_{mor}}
\newcommand{\pic}{\pi\mathbf{Ins}_{co}}

\newcommand{\vertin}{\rotatebox[origin=c]{90}{$\in$}}

\begin{document}
\title{Connecting abstract logics and adjunctions in the theory of ($\pi$-)institutions: some theoretical remarks and applications}


\author{
Gabriel Bittencourt Rios \thanks{Institute of Mathematics and Statistics, University of S\~ao Paulo, Brazil, gabriel.bit@usp.br},
Daniel de Almeida Souza \thanks{Institute of Mathematics and Statistics, University of S\~ao Paulo, Brazil, daniel.almeida.souza@usp.br},\\
Darllan Concei\c c\~ao\ Pinto \thanks{Institute of Mathematics, Federal University of Bahia, Brazil, darllan\underline{ }math@hotmail.com}, 
Hugo Luiz Mariano \thanks{Institute of Mathematics and Statistics, University of S\~ao Paulo, Brazil, hugomar@ime.usp.br}
}

\overfullrule=0pt


\date{}
\maketitle

\begin{abstract}
In the present work, a natural sequel to \cite{MaPi1}, we further discuss the existence of adjunctions between categories of institutions and of $\pi$-institutions. This is done at both a foundational and an applied level. Firstly, we reformulate and conceptually clarify such adjunctions in terms of the $2$-categorical data involved in the construction of categories of institution-like structures. More precisely, we remark that the process used for passing from rooms to institutions (\cite{Diac2}) can be extended, due to its $2$-functoriality, to more general room-like and institution-like structures in such a way that the aforementioned adjunctions are all seen to arise from simpler adjunctions at the room-like level. Secondly, and mostly independently, we provide some applications of such adjunctions to abstract logics, mainly to the setting of propositional logics and filter pairs (\cite{AMP1}); we also generalize the process of skolemization, a classical device from predicate logic, to the institutional setting.









{\bf Keywords:} ($\pi$-)institutions, abstract logics, adjunctions

\end{abstract}

\section*{Introduction}







The concept of \emph{institution} was introduced by J. A. Goguen and R. M. Burstall (see \cite{GB}) in order to present a unified mathematical formalism for the notion of a formal logical system, i.e. it provides a \emph{``...categorical abstract model theory which formalizes the intuitive notion of logical system, including syntax, semantic, and satisfaction relation between them...''} (\cite{Diac2}). This means that it encompasses the abstract concept of universal model theory for a logic: it contains a satisfaction relation between models and sentences that is ``stable under change of notation''. The are several natural examples of institutions, and a systematic study of abstract model theory based on the general notion of institution is presented in Diaconescu's book \cite{Diac2}.

A proof-theoretical variation of the notion  of institution, the concept of \emph{$\pi$-institution}, was introduced by Fiadeiro and Sernadas in  \cite{FS}: it formalizes the notion of a deductive system and  \emph{``...replace the notion of model and satisfaction by a primitive consequence operator (\`a la Tarski)''}. Categories of propositional logics endowed with natural notions of translation morphisms provide examples of $\pi$-institutions. Voutsadakis has developed an intensive study of abstract algebraic logic based on the concept of $\pi$-institution, see for instance \cite{Vou}.

Certain relations between institutions and $\pi$-institutions were established in \cite{FS} and \cite{Vou}. On the other hand, it seems that the explicit functorial connections between the category of institutions (with comorphisms) and that of $\pi$-institutions (with comorphisms) first appeared in \cite{MaPi1}: indeed, the category of $\pi$-institutions is isomorphic to a full coreflective subcategory of the category of institutions. In the present (ongoing) work, we expand the study initiated in \cite{MaPi1} by establishing new adjunctions concerning categories of institution-like structures and sketching new connections between these and abstract logics. Thus the goal of the article is twofold: firstly, a categorical analysis in the setting of the abstract theory of models (respectively, theory of proof) given by institution theory (respectively, $\pi$-institution theory); secondly, applications to presentations of propositional logics (abstract logics, filter pairs) and abstract predicate logic devices (skolemization).

{\bf Overview of the paper:}

In {\bf Section 1} we recall,  for the reader's convenience, the definitions of institution and of $\pi$-institution, as well as their respective notions of (co)morphism. In {\bf Section 2} we expand the work in \cite{MaPi1} by presenting new adjunctions involving categories of categories, diagrams, institutions, and $\pi$-institutions. {\bf Section 3} is devoted to extending the construction of the category of rooms $-$ as presented in \cite{Diac2} $-$ in a way that applies to more general categories of institution-like structures. This is done by applying classical $2$-categorical machinery (such as the $2$-Yoneda embedding and the Grothendieck construction) and, although being relatively straightforward from a technical point of view, its $2$-functoriality allows us to provide a crucial conceptual simplification of the aforementioned adjunctions between categories of institution-like structures: they are seen to arise as images (under a $2$-functor of \emph{institutional realization}) of adjunctions between their generating categories of room-like structures.
In {\bf Section 4}, we present some institutions and $\pi$-institutions of abstract propositional logics, not only the ones  obtained by the former adjunctions, useful for establishing an abstract Glivenko's theorem for algebraizable logics regardless of their signatures associated  (\cite{MaPi3}). We have also defined a  institution for each  filter pair -general and finitary version  (see \cite{AMP1})-  in fact, we provide a functor from the category of filter pairs to the category of institutions that can be restricted to a functor from the category of propositional logics to the category of institutions and, moreover, that  can be extended to a functor from the ``multialgebraic'' setting (logics and filter pairs),  useful to deal with complex logics, as Logics of Formal Inconsistency (LFIs) (\cite{CCM}), thought    non-deterministic semantics of matrices (\cite{AZ}). {\bf Section 5} introduces a new institutional device: skolemization; which is applied to get, by borrowing from FOL, a form of downward L\"owenheim-Skolem for the setting of multialgebras. {\bf Section 6} finishes the paper presenting some remarks and perspectives of future developments.

\section{Preliminaries: categories of institutions and $\pi$-institutions}

In this first section we recall, for the reader's convenience, the definition  of institution and $\pi$-institution with their respective notions of morphisms and comorphisms, consequently defining their categories.  We also add a subsection recalling the main results in \cite{MaPi1}: the adjunction between the categories of institutions and $\pi$-institutions endowed with its {\em comorphisms}.

\subsection{Categories of institutions}
\begin{Df}
An institution $I=(\mathbb{S}ig,Sen,Mod,\models)$ consists of

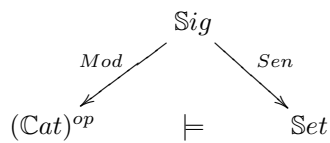
\[\xymatrix{
&\mathbb{S}ig\ar[ld]_{Mod}\ar[rd]^{Sen}&\\
(\mathbb{C}at)^{op}&\models&\mathbb{S}et
}\]

\begin{enumerate}
\item[1.]a category $\mathbb{S}ig$, whose the objects are called {\em signature},
\item[2.]a functor $Sen:\sig\to\set$, for each signature a set whose elements are called {\em sentence} over the signature
\item[3.]a functor $Mod:(\sig)^{op}\to\cat$, for each signature a category whose the objects are called {\em model},
\item[4.]a relation $\models_{\Sigma}\subseteq|Mod(\Sigma)|\times Sen(\Sigma)$ for each $\Sigma\in|\sig|$, called $\Sigma$-{\em satisfaction}, such that for each morphism $h:\Sigma\to\Sigma'$, the compatibility  condition
    \[M'\models_{\Sigma'}Sen(h)(\phi)\ if\ and\ only\ if\ Mod(h)(M')\models_{\Sigma}\phi\]
  holds for each $M'\in|Mod(\Sigma')|$ and $\phi\in Sen(\Sigma)$
\end{enumerate}
\end{Df}

\begin{Ex} \label{inst-examples}
Let $Lang$ denote the  category of  languages $L = ((F_n)_{n \in \N},(R_n)_{n \in \N})$, -- where  $F_n$ is a set of symbols of $n$-ary function symbols and $R_n$ is a   set of symbols of $n$-ary relation symbols, $n \geq 0$ -- and language morphisms\footnote{That can be chosen ``strict" (i.e., $F_n\mapsto F_n'$, $R_n \mapsto R'_n$) or chosen be ``flexible" (i.e., $F_n\mapsto \{n-ary-terms(L')\}$, $R_n \mapsto \{n-ary-atomic-formulas(L')\}$).}. For each pair of cardinals $\aleph_0 \leq \kappa, \lambda \leq \infty$, the category $Lang$  endowed with the usual notion of  $L_{\kappa,\lambda}$-sentences (=  $L_{\kappa,\lambda}$-formulas with no free variable), with the usual association of category of structures and  with the usual (tarskian) notion of satisfaction, gives rise to an institution $I({\kappa,\lambda})$.

\end{Ex}

\begin{Df} Let $I$ and $I'$ be institutions.
\begin{enumerate}
    \item[(a)] An institution {\bf morphism} $h = (\Phi,\alpha,\beta) : I \to I'$ consists of:
    
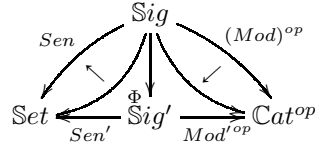
\[\xymatrix{
    &\mathbb{S}ig\ar@/^1pc/[ld]_{\nwarrow}\ar@/_/[ld]_{Sen}\ar@/^/[rd]^{(Mod)^{op}}\ar@/_1pc/^{\swarrow}[rd]\ar[d]_>{\Phi}&\\
    \set&\sig'\ar[l]^{Sen'}\ar[r]_{{Mod'}^{op}}&\cat^{op}
    }\]

    \begin{enumerate}
        \item[$\bullet$]a functor $\Phi:\sig\to\sig'$
        \item[$\bullet$]a natural transformation $\alpha:Sen'\circ \Phi\Rightarrow Sen$
        \item[$\bullet$]a natural transformation $\beta:Mod\Rightarrow Mod'\circ\Phi^{op}$\\

            Such that the following {\em compatibility condition} holds:
            \[m\models_{\Sigma}\alpha_{\Sigma}(\varphi')\ \ if\!f\ \ \beta_{\Sigma}(m)\models'_{\Phi(\Sigma)}\varphi'\]
            For any $\Sigma\in\sig$, any $\Sigma$-model $m$ and any $\Phi(\Sigma)$-sentence $\varphi'$.
    \end{enumerate}
    \item[(b)] A triple $f=\langle \phi,\alpha,\beta\rangle:I\to I'$ is a {\bf comorphism} between the given institutions if the following conditions hold:
    \begin{enumerate}
        \item[$\bullet$]$\phi:\sig\to\sig'$ is a functor.
        \item[$\bullet$]natural transformations $\alpha:Sen\Rightarrow Sen'\circ\phi$ and $\beta:Mod'\circ\phi^{op}\Rightarrow Mod$ satisfying:
        \[m'\models'_{\phi(\Sigma)}\alpha_{\Sigma}(\varphi)\ iff\ \beta_{\Sigma}(m')\models_{\Sigma}\varphi\]
        For any $\Sigma\in\sig$,  $m'\in Mod'(\phi(\Sigma))$ and $\varphi\in Sen(\Sigma)$.
    \end{enumerate}
\end{enumerate}
\end{Df}

Given comorphisms $f:I\to I'$ and $f':I\to I''$, notice that $f'\bullet f \coloneqq  \langle\phi'\circ\phi, \alpha'\bullet\alpha,\beta'\bullet\beta\rangle$ defines a comorphism $f'\bullet f : I \to I''$, where $(\alpha'\bullet\alpha)_{\Sigma}=\alpha'_{\phi(\Sigma)}\circ\alpha_{\Sigma}$ and $(\beta'\bullet\beta)_{\Sigma}=\beta_{\Sigma}\circ\beta'_{\phi(\Sigma)}$. Let $Id_{I} \coloneqq  \langle Id_{\sig},Id,Id\rangle:I\to I$. It is straightforward to check that these data determines a category\footnote{As usual in category theory, the set theoretical size issues on such global constructions of categories can be addressed by the use of at least two Grothendieck universes.}. We will denote by $\inc$ this category of institution comorphisms. Of course, using analagous methods one can also define $\inm$---the category of institution morphisms.

\begin{Ex} \label{inst-morph-examples} 
Given two pairs of cardinals  $(\kappa_i, \lambda_i)$, with  $\aleph_0 \leq \kappa_i, \lambda_i \leq \infty$, $i =0,1$, such that $\kappa_0\leq \kappa_1$ and $\lambda_0  \leq \lambda_1$, then it is induced a morphism and a comorphism of institutions
$(\Phi, \alpha, \beta): I(\kappa_0, \lambda_0) \to I(\kappa_1, \lambda_1)$, given by the {\em same data}: $\sig_{0} = Lang = \sig_{1}$, $Mod_0 = Mod_1 : (Lang)^{op} \to \cat$, $Sen_{i}=L_{\kappa_{i},\lambda_{i}},\ i=0,1$, $\Phi = Id_{Lang} : \sig_0 \to \sig_1$, $\beta := Id : Mod_{i} \Rightarrow Mod_{1-i}$,   $\alpha := inclusion : Sen_{0} \Rightarrow Sen_{1}$.

\end{Ex}

\subsection{Categories of $\pi$-institutions}

\begin{Df}
A $\pi$-institution $J=\langle\sig,Sen,\{C_{\Sigma}\}_{\Sigma\in|\sig|}\rangle$ is a triple with its first two components exactly the same as the first two components of an institution and, for every $\Sigma\in|\sig|$, a closure operator $C_{\Sigma}:\mathcal{P}(Sen(\Sigma))\to\mathcal{P}(Sen(\Sigma))$, such that, for every $f:\Sigma_{1}\to\Sigma_{2}\in Mor(\sig)$, the following holds:
    \[Sen(f)(C_{\Sigma_{1}}(\Gamma))\subseteq C_{\Sigma_{2}}(Sen(f)(\Gamma)),\ for\ all\ \Gamma\subseteq Sen(\Sigma_{1}).\]
\end{Df}

\begin{Df} Let $J$ and $J'$ be $\pi$-institutions.
\begin{itemize}

\item[(a)] A \textbf{morphism} between $J$ and $J'$ is a pair $\langle \Phi, \alpha \rangle$ such that:
    \begin{itemize}
        \item[$\bullet$]$\Phi:\sig\to\sig'$ is a functor
        \item[$\bullet$]$\alpha:Sen'\Phi\Rightarrow Sen$ is a natural transformation\\
        
        And, for all $\ \Gamma\cup\{\varphi\}\subseteq Sen'(\Phi\Sigma)$, the following holds:
        \[\varphi\in C_{\Phi\Sigma}(\Gamma) \Rightarrow \alpha_{\Sigma}(\varphi)\in C_{\Sigma}(\alpha_{\Sigma}(\Gamma)) \]
    \end{itemize}
\item[(b)]$\langle\Phi,\alpha\rangle:J\to J'$ is a \textbf{comorphism} between $\pi$-institution if:
    \begin{itemize}
        \item[$\bullet$]$\Phi:\sig\to\sig'$ is a functor
        \item[$\bullet$]$\alpha:Sen\Rightarrow Sen'\Phi$ is a natural transformation\\
        
        Such that, for all $\ \Gamma\cup\{\varphi\}\subseteq Sen(\Sigma)$, we have:
        \[\varphi\in C_{\Sigma}(\Gamma)\Rightarrow\alpha_{\Sigma}(\varphi)\in C_{\Phi\Sigma}(\alpha_{\Sigma}(\Gamma))\]
    \end{itemize}
\end{itemize}
\end{Df}

Given $\pi$-institution morphisms (respec. comorphisms) $\langle F, \alpha\rangle:J\to J'$ and $\langle G, \beta\rangle:J'\to J''$, $g\cdot f$ is defined as $\langle GF, \alpha \cdot \beta F  \rangle$ (respec. $\langle GF, \beta F \cdot \alpha \rangle$), routine calculations show the composition is well defined. The identity morphism and comorphism are both given by $\langle 1_{\sig}, 1_{Sen} \rangle $. These remarks lead us to define $\pi\mathbf{Ins}_{mor}$ and $\pi\mathbf{Ins}_{co}$ the categories of, respectively, institution morphisms and comorphisms.

\begin{Obs}
It is easy to see that $\pi$-institution can be equivalently described by a triple 
 $\langle\sig,Sen,\{\vdash_{\Sigma}\}_{\Sigma\in|\sig|}\rangle$ where the first two components are simply the ones used for $\pi$-institutions and the third component is a family, indexed by $\Sigma\in|\sig|$, of {\em tarskian consequence relations} $\vdash_{\Sigma}  \ \subseteq \ \mathcal{P}(Sen(\Sigma))\times Sen(\Sigma)$ such that for every arrow $f:\Sigma_{1}\to\Sigma_{2}$ in $\sig$ the induced function $Sen(f) : Sen(\Sigma_{1})\to Sen(\Sigma_{2})\in Mor(\set)$ is a {\em logical translation}, i.e. for each $ \Gamma \cup \{\varphi\} \subseteq Sen(\Sigma_1)$
    $$ \Gamma \vdash_{\Sigma_1} \varphi \ \Rightarrow \ Sen(f)[\Gamma] \vdash_{\Sigma_2} Sen(f)(\varphi)$$
    

\end{Obs}

\subsection{An adjunction between $\inc$\ and $\pic$}\label{adj inc pic}
For the reader's convenience, we recall here the adjunction  between $\inc$ and $\pic$ established in \cite{MaPi1}; thus  all the proofs will be omitted. 

Let $I=\langle\sig,Sen,Mod,\models\rangle$ be an institution. Given $\Sigma\in|\sig|$, consider
\[\Gamma^{\star}=\{m\in Mod(\Sigma);\ m\models_{\Sigma} \varphi\ for\ all\ \varphi \in\Gamma\}\ \text{ and}\]
\[M^{\star}=\{\varphi\in Sen(\Sigma);\ m\models_{\Sigma}\varphi\ for\ all\ m\in M\}\]

for any $\Gamma\subseteq Sen(\Sigma)$ and $M\subseteq Mod(\Sigma)$. Notoriously, these mappings establish a Galois connection. Thus $C^{I}_{\Sigma}(\Gamma) \coloneqq\Gamma^{\star\star}$ defines a closure operator for any $\Sigma\in|\sig|$ (\cite{Vou}). We can now define the first part of our adjunction:






    \[\begin{tikzcd}
\mathbf{Ins}_{co} 	\arrow[r, "F"]                                                      & \pi\mathbf{Ins}_{co}\\
I\arrow[d, "{\langle \phi,\alpha,\beta\rangle}"'{name=A}] \arrow[r, maps to] & {\langle \sig^{I},Sen^{I},\{C^{I}_{\Sigma}\}_{\Sigma\in|\sig|}\rangle} \arrow[d, "{\langle \phi,\alpha\rangle}"{name=B}] \\
J \arrow[r, maps to]                                    &	{\langle \sig^{J},Sen^{J},\{C^{J}_{\Sigma}\}_{\Sigma\in|\sig|}\rangle}                             \arrow[from=A,to=B,maps to, shorten <= 2.2mm, shorten >= 20mm]
\end{tikzcd}\]

For the other side of the adjunction consider the application:

    \[\begin{tikzcd}
\pi\mathbf{Ins}_{co}\arrow[r, "G"]                                                      & \mathbf{Ins}_{co} \\
J\arrow[d, "{\langle \phi,\alpha\rangle}"'{name=A}] \arrow[r, maps to] & \langle\sig^{J},Sen^{J},Mod^{J},\models^{J}\rangle \arrow[d, "{\langle \phi,\alpha,\alpha^{-1}\rangle}"{name=B}] \\
J' \arrow[r, maps to]                                    &	\langle\sig^{J'},Sen^{J'},Mod^{J'},\models^{J'}\rangle                           \arrow[from=A,to=B,maps to, shorten <= 2.7mm, shorten >= 20mm]
\end{tikzcd}\]

Where:
\begin{itemize}
    \item $Mod^{J}$ is taken as:
    \[\begin{tikzcd}
\sig^{op}	\arrow[r, "Mod^{J}"]                                                      & \cat\\
\Sigma \arrow[d, "f"'{name=A}] \arrow[r, maps to] &\{C_{\Sigma}(\Gamma):\ \Gamma\subseteq Sen(\Sigma)\} \\
\Sigma' \arrow[r, maps to]                                    &	\{C_{\Sigma'}(\Gamma):\ \Gamma\subseteq Sen(\Sigma')\} \arrow[u,swap,"Sen(f)^{-1}"{name=B}]                           \arrow[from=A,to=B,maps to, shorten <= 3mm, shorten >= 20mm]
\end{tikzcd}\]
With $Mod^{J}(\Sigma)$ being viewed as a ``co-discrete category''\footnote{I.e., a class of objects  $C$ endowed with the trivial groupoid structure of all ordered pairs, $C \times C$.}.
    \item  For each $\Sigma$ we let $\models^{J}_\Sigma\subseteq|Mod(\Sigma)|\times Sen(\Sigma)$ as the relation:
\[m\models^{J}_{\Sigma}\varphi\quad if\hspace{-0.1cm}f \quad\varphi\in m\]
For any $m\in Mod(\Sigma)$ and $\varphi\in Sen(\Sigma)$
\end{itemize}

\begin{Teo}
The functors $F : \inc \rightarrow \pic$ and $G : \pic\rightarrow \inc$ defined above establish an adjunction $G \dashv F$ between the categories $\inc$ and $\pic$. Moreover, $F \circ G = Id_{\pic}$ and the unity of this adjunction, the natural transformation $\eta : Id_{\pic} \rightarrow F\circ G$, is the identity. Thus the category ${\pic}$ can be seen to be a full coreflective subcategory of $\inc$.
\end{Teo}

\section{Adjunctions between Inst, $\pi$-Inst, Cat, Diag}

In this section we continue and expand the analysis of categorical relations between categories whose objects are categories endowed with some extra structure like categories of  ($\pi$-)institutions, categories of categories and categories of $Set$-based diagrams.

\subsection{An adjunction between $\inm$ and $\pim$}\label{adj inm pim}


It is natural to ask whether we could achieve a similar adjunction considering morphisms instead of comorphisms, that is, taking $\inm$ and $\pim$ instead of $\inc$ and $\pic$. 
In  this subsection, we sketch a proof that the category of $\pi$-institutions and morphisms is isomorphic to a full coreflective subcategory of the category of institutions and morphisms: this is  a natural variant of the results in \cite{MaPi1} which were recalled in subsection \ref{adj inc pic}.

Let $I=\langle \sig,Sen,Mod,\models \rangle$ be an institution. Given $\Sigma\in|\sig|$ let:
\[\Gamma^*:=\{m\in Mod(\Sigma):m\models_{\Sigma}\varphi \ for \ all \ \varphi\in\Gamma\} \] \[and\] \[
M^* := \{\varphi\in Sen(\Sigma): m \models_{\Sigma}\varphi \ for \ all \ m\in M\}
\]
for any $\Gamma\subseteq Sen(\Sigma)$ and $M\subseteq |Mod(\Sigma)|$. These mappings cleary define a Galois connection between $\mathcal{P}(Sen(\Sigma))$ and $\mathcal{P}(|Mod(\Sigma)|)$. Therefore, $Con^{I}_\Sigma(\Gamma):=\Gamma^{**}$ defines a closure operator on $\mathcal{P}(Sen(\Sigma))$ for any $\Sigma\in|\sig|$.
\begin{Lem}
    Let $\langle \phi,\alpha,\beta\rangle:I\to I'$ be an arrow in $\inm$ and $\sigma\in|\sig|$. Given $\Gamma\subseteq Sen(\Sigma)$ and $M\subseteq |Mod(\Sigma)|$ the following holds:
    \begin{itemize}
        \item $\beta_{\Sigma}[(\alpha_{\Sigma}[\Gamma])^*]\subseteq\Gamma^*$
        \item $\alpha_{\Sigma}[(\beta_{\Sigma}[M])^*]\subseteq M^*$
    \end{itemize}
\end{Lem}
\Dem
The proof is similar to the one of \textbf{Lemma 2.8} in \cite{MaPi1}\\
\qed

Consider now the following functor:
\begin{align*}
    F:\inm&\to\pim\\
        I &\mapsto \langle\sig,Sen,\{Con^I_\Sigma\}_{\Sigma\in|\sig|}\rangle
\end{align*}
The proof that F is well defined on objects can be found on \cite{MaPi1}. The action on morphisms is defined as follows:
\begin{align*}
    I&\xrightarrow{\langle\phi,\alpha,\beta\rangle}I'\\
    F(I)&\xrightarrow{\ \langle\phi,\alpha\rangle}F(I')
\end{align*}
Consider now the following application,
\begin{align*}
    G:\pim&\to\inm\\
    J&\to\langle \sig,Sen,Mod^J,\models^J\rangle
\end{align*}
Where:
\begin{itemize}
    \item $Mod^J:\sig^{op}\to\cat$ is defined as: 
\[\Sigma\xrightarrow{f}\Sigma'\mapsto\{C_{\Sigma'}(\Gamma'):\Gamma'\subseteq Sen(\Sigma')\}\xrightarrow{Sen(f)^{-1}}\{C_\Sigma(\Gamma):\Gamma\subseteq Sen(\Sigma)\}\]
    \item For each $\Sigma\in|\sig|$, $\models^{J}_{\Sigma}\subseteq|Mod^J(\sigma)|\times Sen(\Sigma)$ is defined such that, give $m\in |Mod(\Sigma)|$ and $\varphi \in Sen(\sigma)$, $m\models^{J}_{\Sigma}\varphi$ iff $\varphi\in m$.\\
The proof that $Mod^J$ is well defined and that $G(J)$ satisfies the compatibility condition and is indeed an institution can be found in \cite{MaPi1}
\end{itemize}
Given a morphism $f=\langle \phi,\alpha\rangle:J\to J'\ $ in $\ \pim$ define, for $\Sigma\in|\sig|$ and $m\in |Mod^{J}(\Sigma)|$, $\beta_\Sigma(m):=\alpha_{\Sigma}^{-1}(m)$. Let us prove that $\beta_\Sigma:Mod^{J}(\Sigma)\to Mod^{J'}(\phi(\Sigma))$.
\[\xymatrix{
\mathcal{P}(Sen(\Sigma)) & \mathcal{P}(Sen'(\phi(\Sigma)))\ar[l]_{\alpha_\Sigma^{-1}}\\
\mathcal{P}(Sen(\Sigma'))\ar[u]^{Sen(f)^{-1}} & \mathcal{P}(Sen'(\phi(\Sigma')))\ar[l]^{\alpha_{\Sigma'}^{-1}}\ar[u]_{Sen('\phi(f))^{-1}}    }\]
Let us register prove the compatibility condition for morphisms. Given $\Sigma\in|\sig|$, $m\in Mod^J(\Sigma)$ and $\varphi\in Sen(\phi(\Sigma))$ we have:
\begin{align*}
    m\models^J_\Sigma\alpha_\Sigma(\varphi) &\iff \alpha_\Sigma(\varphi)\in m\\
                                            &\iff \varphi\in\alpha_\Sigma^{-1}(m)\\
                                            &\iff \varphi\in\beta_\Sigma(m)\\
                                            &\iff \beta_\Sigma(m)\models^{J'}_{\phi(\Sigma)}\varphi
\end{align*}
It follows that $G(f)=\langle \phi,\alpha,\beta\rangle$ is a morphism of institutions. To prove $G$ a functor simply notice that, given $f=\langle\phi,\alpha\rangle:J\to J'$ and $f'=\langle\phi',\alpha'\rangle:J'\to J''$ in $\pim$, $G(f'\cdot f)=\langle \phi'\cdot \phi,\alpha'\cdot\alpha \phi,(\alpha'\cdot\alpha \phi)^{-1}\rangle = \langle \phi'\cdot \phi,\alpha'\cdot\alpha \phi, \alpha^{-1} \phi\cdot\alpha'^{-1} \rangle = G(f')\cdot G(f)$ and, for any $\pi$-institution J, routine calculations show $G(1_J)=1_{G(I)}$.

In fact, as in \cite{MaPi1}, we have the following:

\begin{Teo} 
    The functors \begin{tikzcd} \inm \arrow[r,"F",shift left = 1] & \pim \arrow[l,"G",shift left = 1] \end{tikzcd} establish and adjunction $G\dashv F$. Moreover, since $F \circ G = Id_{\pim}$ and the unity of this adjunction, the natural transformation $\eta : Id_{\pim} \rightarrow F\circ G$, is the identity. Thus the category ${\pim}$ can be seen as a full coreflective subcategory of $\inm$.
\end{Teo}

\subsection{Adjunctions between $\CAT$ and $\pic$}\label{adj cat pic}

In this section we detail left and right adjoints for the forgetful functor from $\pic$ to $\CAT$. Something of notice here is the similarity between the functors shown here and the adjoints to the forgetful functor from {\bf{Top}} to $\set$. Indeed, we describe a left adjoint that associates categories to their ``discrete" $\pi$-institution, where every set is closed, and a right adjoint that maps to their ``codiscrete" $\pi$-institution, where the only closed sets are the empty set and the entire set of formulas. The place of these two constructions in the theory of $\pi$-institutions is then similar to the place of the ``(co)discrete" topology in point set topology. That is to say, as illustrative examples of pathologies.

Let us commence by the right adjoint. We begin by defining an action on the objects of $\CAT$; given a category $\mathcal{A}$ let $\top\mathcal{A}:=\langle\mathcal{A},\ast,\{Con_{c}\}_{a\in|\mathcal{A}|}\rangle$ where $\ast:\mathcal{A}\to\set$ is the constant functor to the singleton set and, for each object $a$ in $\mathcal{A}$ and $\Gamma\subseteq\{\ast\}$, we define $Con_{a}(\Gamma)=\{\ast\}$. It is clear that $Con_{a}$ is closure operator on $\{\ast\}$. Moreover, for any arrow $a\xrightarrow{f}a'$ in $\mathcal{A}$ and $\Gamma\subseteq\{\ast\}$, we have that $\ast f(Con_{a}(\Gamma))=Con_{a'}(\ast f(\Gamma))$ and thus $\top\mathcal{A}$ is a $\pi$-institution.

We can now extend $\top$ to morphisms. Given some functor $F:\mathcal{A}\to\mathcal{B}$, we see that there is a unique $!: \ast \Rightarrow\ast F$; furthermore, routine calculations show $\varphi\in Con_{a}(\Gamma)\Rightarrow !_{a}(\varphi)\in Con_{Fa}(!_{a}(\Gamma))$ for $\{\varphi\}\cup\Gamma\subseteq\{\ast\}$. Define then $\top F=\langle F,!\, \rangle$ the remarks above showing it a comorphism between $\top\mathcal{A}$ and $\top\mathcal{B}$.

To prove that $\top$ behaves functorially notice, firstly, that the lone arrow $\ast\Rightarrow\ast$ is $1_{\ast}$ so $\top(1_{\mathcal{A}})=\langle 1_{\mathcal{A}},1_{\ast}\rangle=1_{\top\mathcal{A}}$. Finally, the below diagram guarantees that the composition is well behaved.

\[\xymatrix{
\ast c\ar@{-->}[r]\ar@{-->}[d]& \ast Fc\ar@{-->}[r]\ar@{-->}[d]&\ast GFc\ar@{-->}[d]\\
\ast c'\ar@{-->}[r]&\ast Fc'\ar@{-->}[r]&\ast GFc'
}\]

\begin{Teo}
    Let $\ U:\pic\to\CAT$ the forgetful functor, taking each $\pi$-institution to its signature category and each comorphism to its first coordinate. The functors $\top:\CAT\to\pic$ and $U:\pic\to\CAT$ establish an adjunction $\top\vdash U$ with counit $\eta_{\mathcal{A}}= 1_{\mathcal{A}}$.
\end{Teo}
\Dem
Given some a $\pi$-institution J and a functor $F:\sig^J\to\mathcal{A}$, consider the below diagram:
\[\begin{tikzcd}
\mathcal{A} & U\top\mathcal{A}\arrow[l, "1_{\mathcal{A}}"'] &  & \top\mathcal{A} \\
                                                          & \mathbb{S}ig^{J}\arrow[lu, "F"]\arrow[u, "F"'{name=U}]                &  & J\arrow[u, "{\langle F,\alpha\rangle}"{name=D}]
                                                          \arrow[from=D,to=U,shorten= 4mm, xshift=2ex, mapsto]
\end{tikzcd}\]
Where $\alpha$ is the single arrow $Sen\Rightarrow\ast F$. Given $\{\varphi\}\cup\Gamma\subseteq Sen(\Sigma)$ we have that $\varphi\in C_{\Sigma}(\Gamma)\Rightarrow\alpha_{\Sigma}(\varphi)=\ast$. As $Con_{F\Sigma}(\alpha_{\Sigma}(\Gamma))=\{\ast\}$ it follows that $\varphi\in C_{\Sigma}(\Gamma)\Rightarrow\alpha_{\Sigma}(\varphi)\in Con_{F\Sigma}(\alpha_{\Sigma}(\Gamma))$ and thus $\langle F,\alpha\rangle$ is indeed a comorphism between $J$ and $D\mathcal{A}$. As $\langle F,\alpha\rangle$ is clearly the only arrow that makes the diagram commute, the result follows.
\qed

We can now describe the left adjoint. Consider the following functor:

 \[\begin{tikzcd}
\bot:\CAT \arrow[r]                            & \pi\mathbf{Ins}_{co}                                                                                                \\
\mathcal{A} \arrow[d, "F"'{name=A}] \arrow[r, maps to] & {\langle \mathcal{A},\emptyset, (Con_a)_{a\in|\mathcal{A}|} \rangle} \arrow[d, "{\langle F,!\,\rangle}"{name=B}] \\
\mathcal{B} \arrow[r, maps to]                 & {\langle \mathcal{B},\emptyset, (Con_b)_{b\in|\mathcal{B}|}\rangle}
\arrow[from=A,to=B,shorten <= 3mm, shorten >= 16.5mm,maps to]
\end{tikzcd}\]

Where $\emptyset$ is the constant functor to the empty set, $Con_{a}$ is the single closure operator on the empty set and $!$ is the unique natural transformation $\emptyset\Rightarrow\emptyset F$. By vacuity, $\langle F,!\,\rangle$ satisfies the comorphism condition. Proving that $\bot$ is indeed a functor uses similar arguments to the ones given above.

\begin{Teo}
    Let $U$ as above. The functors $\bot$ and $U$ establish an adjunction $\bot\dashv U$ with unit $\epsilon_{\mathcal{A}}= 1_{\mathcal{A}}$.
\end{Teo}
\Dem 
Given some a $\pi$-institution J and a functor $F:\mathcal{A}\to\sig^J$, consider the below diagram:
\[\begin{tikzcd}
\mathcal{A}\arrow[r, "1_{\mathcal{A}}"]\arrow[rd, "F"'] & U\bot\mathcal{A}\arrow[d, "F"{name=U}]  &  & \bot\mathcal{A}\arrow[d, "{\langle F,\alpha\rangle}"'{name=D}]\\
                                                          & \mathbb{S}ig^{J}               &  & J
                                                          \arrow[from=D,to=U,shorten= 4mm, xshift=2ex, mapsto]
\end{tikzcd}\]
Where $\alpha$ is the only natural transformation $\emptyset\Rightarrow Sen^{J}F$. We argue by vacuity to show that $\langle F,\alpha\rangle$ is a comorphism. Since $\langle F,\alpha\rangle$ it is clearly the only arrow that makes the diagram commute, the result follows.\\
\qed
\begin{Obs}
    It is easy to see how one would go on defining the $\pim$ versions of the functors $\top$ and $\bot$. This, of course, prompt us to question if these functors still define an adjunction. Routine calculations show that the directions would be reversed, that is, in the $\pim$ case we have: $\bot\vdash U \vdash\top$
\end{Obs}
\begin{Obs}\label{ExemploDeDiag}
    Let us consider a generalization of $\pic$ for a moment. Given a concrete category $\mathcal{C}$, i.e. a faithful functor $|-| : \mathcal{C} \to Set$, a  $\mathcal{C}\!-\!\pi\!-\!$institution is a triple of the form $\langle \sig, Sen:\sig\to C, (C_\Sigma:\mathcal{P}|Sen(\Sigma)|\to\mathcal{P}|Sen(\Sigma)|)_{\Sigma\in|C|} \rangle$ where $\sig$ is a category, $Sen$ a functor and $C_\Sigma$ a closure operator on $\mathcal{P}|Sen(\Sigma)|$ satisfying structurality; furthermore, one can easily generalize a version of comorphisms for $\mathcal{C}\!-\!\pi\!-\!$institutions. Consider then $\mathcal{C}\!-\!\pic$--- the category of $\mathcal{C}\!-\!\pi\!-\!$institution comorphisms.\\
    Let $1$ a terminal object in the concrete category $\mathcal{C}$. We can now define a functor $\top_\mathcal{C}:\CAT\to\mathcal{C}\!-\!\pic$ as 
    \[\mathcal{A}\xrightarrow{F}\mathcal{B}\mapsto\langle\mathcal{A},1,(Con_a)_{a\in Ob(\mathcal{A})}\rangle\xrightarrow{\langle F,\alpha\rangle}\langle\mathcal{B},1,(Con_b)_{b\in Ob(\mathcal{B})}\rangle\]
    Where $1$ is the constant functor to the terminal object, $Con_a(\Gamma)=|Sen(a)|$ for each $a\in Ob(\mathcal{A})$ and $\Gamma\subseteq|Sen(a)|$ and $\alpha$ is the unique $1\Rightarrow1F$. Using the methods analogous we see that $\top_\mathcal{C}\vdash forgetful$. Suppose now that $\mathcal{C}$ had a initial object $0$, one can easily see how to define $\bot_\mathcal{C}$ --- the left adjoint to the forgetful --- mimicking $\bot$.\\
    It is common, specially when dealing with propositional logics, to define the syntax as an algebraic structure instead of a set. This remark could be of use in that scenario.
\end{Obs}

\subsection{Adjunctions $Diag_{co} \leftrightarrows \pic$}




We begin this section by describing $Diag_{co}(C)$ and $Diag_{mor}(C)$, the categories of diagrams for a given category $C$. Diagrams for $\bf{Set}$ can be initially seen as $\pi$-institutions minus the consequence relation and the $2$-categorially minded will recognize diagrams for $C$ as the Grothendieck construction for $\mathbf{CAT}(-,C)$. After this introduction, we proceed to obtain right and left adjoints to the the forgetful $Diag_{co}(\bf{Set})\to\pic$. Finally, we further this result to categories adjoint to $\bf{Set}$. In this sense the purpose of this section is twofold:
\begin{itemize}
\item Firstly, it may serve as a path to the theory of ``generalized" $\pi$-institutions, that is, $\pi$-institutions having sentence functors over any arbitrary category, not only $\bf{Set}$. This practice of taking sentences in categories different of $\bf{Set}$ is common in logic, a notorious example being that of propostional logic where sentences are taken as free algebras.\\
\item Secondly, it introduces, albeit tacitly, the $2$-categorial ideas which will be used in the next section. Indeed, the idea of diagrams will be explored again in section \ref{Categories of institutions as Grothendieck categories}.
\end{itemize}

Let $C$ be a category. Denote $Diag_{co}(C)$ the category whose objects are pair $(A,F)$, where $F : A \to C$  is a covariant functor and such that $Hom((A,F), (A', F'))$ is the (meta)class of all pairs $(T, \alpha)$ where $T : A \to A'$ is a functor and $\alpha : F \rightarrow F' \circ T $ is a natural transformation. Let $id_{(A,F)} := (id_A,id_F)$ and if $(T', \alpha') \in Hom((A',F'), (A'', F''))$, then $(T',\alpha') \bullet (T, \alpha) := (T' \circ T, \alpha'_T \circ \alpha)$. $Diag_{mor}(C)$ denotes the category with the same objects as $Diag_{co}(C)$ and, for arrows, $(T, \alpha) \in Hom((A,F), (A',F'))$ iff $T : A \to A'$ is a functor and $\alpha : F' \circ T \to F$ is a natural transformation; identities are the same as in $Diag_{co}(C)$ and compositions are adapted accordingly: $(T',\alpha') \bullet (T, \alpha) := (T' \circ T, \alpha \circ \alpha'_T)$. 

Now consider the category $\pic$ and the obvious forgetful functor $U : \pic \to Diag_{co} (Set)$ given by:

\[\begin{tikzcd}
\pi\mathbf{Ins}_{co} \arrow[r, "U"] & Diag_{co}(Set)                                                           \\
{\langle \sig, Sen, (C_\Sigma)_{\Sigma\in|\sig|}\rangle} \arrow[d, "{\langle F,\alpha\rangle}"'{name=A}] \arrow[r, maps to,shorten <= 1.3mm,shorten >= 0.5mm] & {\langle\sig,Sen\rangle} \arrow[d, "{\langle F,\alpha \rangle}"{name=B}] \\
{\langle \sig', Sen',(C'_\Sigma)_{\Sigma\in|\sig|}\rangle} \arrow[r, maps to]                                      & {\langle\sig',Sen' \rangle}
\arrow[from=A,to=B,mapsto, shorten <= 19.5mm, shorten >= 10mm]
\end{tikzcd}\]

The main result of this subsection is that $U$ has a left adjoint $L : Diag_{co}(Set) \to \pic$ and a right adjoint $R : Diag_{co}(Set) \to \pic$. Thus $U : \pic \to Diag_{co}(Set)$ preserves all limits and all colimits.

We will provide just the definitions of the functors, since the proof of the universal properties are straightforward.

$L : Diag_{co}(Set) \to \pic$ is given by:
$L(A,F) := (A, F, (C^{min}_a)_{a \in |A|})$, where $C^{min}_a : P(F(a)) \to P(F(a))$ is such that:

$$ \Gamma \in P(F(a)) \mapsto C^{min}_a(\Gamma) := \Gamma$$ 

It is ease to see that $L(A,F)$  satisfies the coherence condition in the definition of $\pi$-institution.

The action of $L$  on morphisms is very simple: 
$$L(((A,F) \overset{(T, \alpha)}\to \ (A',F'))  \ = \ (A, F, (C^{min}_a)_{a \in |A|}) \overset{(T, \alpha)}\to \ (A',F', ({C'}^{min}_{a'})_{a' \in |A'|});$$

this clearly determines a morphism of $\pi$-institutions.

For each $(A,F) \in|Diag_{co}(Set)|$, we have the identity arrow $id_{(A,F)} : (A,F) \to U(L(A,F))$ and this is a initial object in the comma category $(A,F) \downarrow U$. Thus $L$ is left adjoint to $U$ and we have just described the component $(A,F)$ of the unity of this adjunction.

Similarly, we have a functor $R : Diag_{co}(Set) \to \pic$ with action $R(A,F) \coloneqq (A, F, (C^{max}_a)_{a \in |A|})$, where $C^{max}_a : P(F(a)) \to P(F(a))$ is such that:

$$ \Gamma \in P(F(a)) \mapsto C^{max}_a(\Gamma) \coloneqq F(a)$$

With the obvious action on arrows, $R$ becomes the right adjoint to $U$.\\

\begin{Obs}\label{smth}
Given category $C$ and a functor $C\xrightarrow{E}\mathbf{Set}$ with left adjoint $\mathbf{Set}\xrightarrow{\mathcal{L}}C$ (respec. right adjoint $\mathbf{Set}\xrightarrow{\mathcal{R}}C$) we can form $Diag_{co}(C) \xrightarrow{\tilde{E}} Diag_{co}(\mathbf{Set})$ and $Diag_{co}(\mathbf{Set})\xrightarrow{\tilde{L}}Diag_{co}(C)$ by composing: $$\tilde{E}((T, \alpha) : (A,F) \to (A',F')) = (T, E\alpha)$$ and likewise for $\tilde{\mathcal{L}}$ (respec. $\tilde{\mathcal{R}}$). It is straightforward that $\tilde{E}$ has as left adjoint $\tilde{\mathcal{L}}$ (respec. right adjoint $\tilde{\mathcal{R}}$). We can then compose this adjunction with the one obtained above to obtain $\begin{tikzcd}
            \pi\mathbf{Ins}_{co}\arrow[r, shift left=1ex, "\tilde{E}\circ U"{name=G}] & Diag_{co}(C) \arrow[l, shift left=.5ex, "L\circ \tilde{\mathcal{R}}"{name=F}]
            \arrow[phantom, from=F, to=G, , "\scriptscriptstyle\boldsymbol{\top}"].
        \end{tikzcd}$ (respec. $\begin{tikzcd}
            Diag_{co}(C)\arrow[r, shift left=1ex, "R\circ \tilde{E}"{name=G}] & \pi\mathbf{Ins}_{co}\arrow[l, shift left=.5ex, "\tilde{\mathcal{L}}\circ U"{name=F}]
            \arrow[phantom, from=F, to=G, , "\scriptscriptstyle\boldsymbol{\top}"]
        \end{tikzcd}$).
\end{Obs}








{We summarize below the adjunctions previously presented. It can be described an analogous diagram for ``morphisms''  instead of ``co-morphism''.}

\[\xymatrix{
&\bf Inst_{co}\ar@<0ex>[d]&\\
\bf CAT\ar@<0ex>[r]&\bf \pi-Inst_{co}\ar@<1ex>[u]\ar@<1ex>[l]\ar@<0ex>[d]&\\
&Diag_{co}(Set)\ar@<1ex>[u]\ar@<-1ex>[u]\ar@{-->}[r]<1ex>\ar@{-->}[r]<-1ex>& Diag_{co}(C)\ar@<0ex>[l]
}\]

\section{Adjunctions at the level of room-like structures}



Accordding to \cite{Diac2}, page 47,

{\em ``The presentation of institutions as functors was given already in \cite{GB} and the 2-categorical structure of the category of institutions has been studied in \cite{Diac1} ."}

This section aims at describing how a standard construction from $2$-category theory $-$ the Grothendieck construction, which associates a fibration to a pseudofunctor in a $2$-functorial way $-$ allows us to reformulate the above adjunctions between categories of institution-like structures in a way which is general and systematic, and which provides conceptually clearer equivalent descriptions of the same phenomena. This is done in two main steps:
\begin{enumerate}
    \item We borrow from \cite{Diac2} the definition of the \emph{category of rooms}, denoted by $\Room$ $-$ which can be used to provide a concise description of the category of institutions\footnote{In \cite{Diac2}, this description is used to show that $\inm$ is a complete category.} $-$ and generalize it in a straightforward way (to categories of room-like structures), so as to obtain analogous reconstructions of categories of institution-like structures.
    
    \item By using the (non-trivial) facts that (a) the process of associating fibrations to pseudofunctors defines a $2$-categorical equivalence, and (b) the $2$-categorical Yoneda embedding is $2$-fully faithful, we are able to conclude that the $2$-functorial procedure (described below) which sends categories of room-like structures to categories of institution-like structures is also $2$-fully faithful. As a corollary, any $2$-categorical connections between categories of institution-like objects can be "pulled-back" to a corresponding construction at the level of room-like structures. For the purposes of this paper, we shall only be concerned with the particular case of recovering instutition-level adjunctions in terms of much simpler room-level adjunctions.
\end{enumerate}

The definition of categories of room-like objects is illustrated in terms of three archetypal examples: for institutions (as usual), for $\pi$-institutions (a direct analogous), and for the category of small categories (which turns out to be an extremal example).

It should also be remarked that the aforementioned procedure comes naturally, and quite generally, in two variants: one suitable for describing \emph{morphisms} between institution-like structures, and one suitable for describing \emph{comorphisms} between them.


Before introducing the actual definitions, we outline as follows the background to be considered: as described in \cite{Diac2}, the category of institutions and morphisms can be obtained by means of a standard categorical construction often referred to as the \textit{Grothendieck construction}. There, a central role is played by the so-called \textit{category of rooms}, denoted by $\Room$: individually, an institution having $\sig$ as its category of signatures corresponds to a functor $\sig \longrightarrow \Room$; on the other hand, (co)morphisms of institutions should also take into account base-change functors between different categories of signatures. The Grothendieck construction provides an adequate framework for studying this kind of phenomena. More precisely, given a $1$-category $\mathcal C$ (regarded as a strict $2$-category with trivial $2$-cells), the Grothendieck construction, which we shall denote by $-^\sharp$, associates to each pseudofunctor $F:\mathcal C \longrightarrow \CAT$ a $1$-category $F^\sharp$ together with a structure (\textit{projection}) functor $F^\sharp \longrightarrow \mathcal C$ onto the base category. Most importantly, it constitutes a pseudofunctor

$$
-^\sharp:[\mathcal C,\CAT] \longrightarrow \CAT/\mathcal C,
$$

where:

\begin{itemize}
    \item $[\mathcal C,\CAT]$ denotes the $2$-category of pseudofunctors $\mathcal C \longrightarrow \CAT$, pseudonatural transformations, and modifications.
    
    \item $\CAT/\mathcal C$ denotes the slice $2$-category defined in the obvious way.
\end{itemize}

Our main interest will be the case where $\mathcal C$ is $\cat$, the $1$-category of categories. We shall also need to consider the $2$-categorical Yoneda (pseudo)functor

\begin{align*}
    Y:\textbf{C} & \longrightarrow [\textbf{C}^{op},\CAT]\\
    c & \longmapsto \textbf{C}(-,c)
\end{align*}

associated to a (possibly weak) $2$-category $\textbf{C}$, and variations thereof. A pseudofunctor equivalent to one of the form $\textbf{C}(-,c)$ is called a \textit{representable $2$-presheaf}. We will be concerned with (restrictions to $\CAT$ of) $2$-presheaves on a (suitably large) $2$-category of categories which are represented by variations of $\Room$. For instance, $\inc$ is described in \cite{Diac2} as the Grothendieck construction $\CAT(-^{op},\Room)^\sharp$ of the Yoneda-like $2$-presheaf $\CAT(-^{op},\Room)$ on $\CAT$. Our goal in this section will be to provide an alternative description of the above adjunctions between categories of institution-like structures (such as institutions and $\pi$-institutions), by noticing that (i) it is easy to describe $\Room$-like categories from which other categories of institution-like structures can be obtained through a similar Yoneda-followed-by-Grothendieck procedure, and (ii) the notion of adjunction is available for any $2$-category, and adjunctions in this sense are preserved by pseudofunctors.

As for categorical prerequisites, we restrict ourselves to providing quick (and mostly ad-hoc) descriptions of some of the necessary constructions from $2$-category theory, including the Grothendieck construction; hence the reader is strongly encouraged to have a prior basic knowledge on these topics. For that purpose, we refer to \cite{Diac2} and \cite{nLab} for a brief introduction, and to \cite{Jo} for a more detailed discussion.

The present section does not aim at completeness; instead, it consists in a brief introduction, including basic constructions a few functioning examples, to the idea of canonically producing new (resp. recovering well-known) $2$-categorical information on categories of institution-like structures in terms of their simpler counterparts: categories of room-like structures.

\subsection{$2$-categorical preliminaries}

We start by fixing some notations and defining the $2$-categorical constructions alluded to above. The basic language of $2$-category theory will be freely used. Unless otherwise specified, by a $2$-category we mean a \textit{strict} $2$-category. If $\mathcal C$ is a $1$-category, we regard it as a $2$-category whenever necessary. We denote by $\CAT$ the $2$-category of categories, functors, and natural transformations, and by $\cat$ the $1$-category of categories and functors. Given $2$-categories $\textbf C$ and $\textbf D$, we denote by $[\textbf C, \textbf D]$ the corresponding category of pseudofunctors, pseudonatural transformations, and modifications. If $\textbf C$ is a $2$-category, we denote by $\textbf C^{op}$ (resp. $\textbf C^{co}$, $\textbf C^{coop}$) the $2$-category obtained by reversing the $1$-cells (resp. $2$-cells, both $1$-cells and $2$-cells). By a contravariant pseudofunctor from $\textbf C$ to $\textbf D$ we mean a pseudofunctor $\textbf C^{op} \longrightarrow \textbf D$. By a $2$-presheaf (resp. category of $2$-presheaves) we mean a pseudofunctor $\textbf C^{op} \longrightarrow \CAT$ (resp. a $2$-category $[\textbf C^{op},\CAT]$).

\subsubsection{The Grothendieck construction}

The Grothendieck construction can be defined in two similar versions: taking as input either a contravariant $\CAT$-valued pseudofunctor (i.e. a $2$-presheaf), or a covariant one.

\begin{definition}
(Grothendieck construction for contravariant pseudofunctors)
\end{definition}
Let $\mathcal C$ be a $1$-category. Given a pseudofunctor $F:\mathcal C^{op} \longrightarrow \CAT$, we define its \textit{Grothendieck construction} or \textit{Grothendieck category}, denoted by $F^\sharp$, as the $1$-category given by the following data:

\begin{itemize}
    \item Its objects are pairs $(c,x)$, where $c \in Ob(\mathcal C)$ and $x \in Ob(F(c))$.
    
    \item An arrow $(c,x) \longrightarrow (d,y)$ is a pair $(f,\phi)$, where $f \in \mathcal C(c,d)$ and $\phi \in F(c)(x,Ff(y))$.
    
    \item The composite of morphisms $(f,\phi):(c,x)\longrightarrow (d,y)$ and $(g,\psi):(d,y)\longrightarrow (e,z)$ is defined as
    
    $$(g \circ f\text{  },\text{  } \alpha^{f,g}_z \circ F(f)(\psi)\circ \phi),$$
    
    where $\alpha^{f,g}$ is the natural isomorphism (associated to $F$ by the definition of a pseudofunctor) $F(f) \circ F(g) \implies F(g \circ f)$. See
    
    \[
    \begin{tikzcd}
    x \arrow{r}{\phi} & Ff(y) \arrow{r}{Ff(\psi)} & Ff(Fg(z))\! =\! (Ff \circ Fg)(z) \arrow{r}{\alpha^{f,g}_z} & F(g \circ f)(z).
    \end{tikzcd}
    \]
    
\end{itemize}

The reader will be able to check that composition is associative and that each object possesses an identity arrow (by using the natural isomorphisms $\alpha^c:1_{F(c)}\implies F(id_c)$). The category $F^\sharp$ is canonically endowed with a (\textit{projection}) functor $F^\sharp \longrightarrow \mathcal C$ given by $(c,x) \longmapsto c$ and $(f,\phi) \longmapsto f$.

Now, suppose given a $1$-cell in $[\mathcal C^{op}, \CAT]$, i.e. a pseudonatural transformation $\eta:F \implies G$. We define a functor $\eta^\sharp:F^\sharp \longrightarrow G^\sharp$ as follows:

\begin{itemize}
    \item $\eta^\sharp((c,x))=(c,\eta_c(x))$ for each $(c,x) \in Ob(F^\sharp)$.
    
    \item For each $(f,\phi):(c,x) \longrightarrow (d,y)$ in $F^\sharp$, we define $\eta^\sharp((f,\phi)):(c,\eta_c(x)) \longrightarrow (d,\eta_d(y))$ as
    
    $$(f \text{  },\text{  } \gamma^f_y \circ \eta_c(\phi)),$$
    
    where $\gamma^f$ is the natural isomorphism (associated to $\eta$ by the definition of a pseudonatural transformation) as in
    
    \[
    \begin{tikzcd}
    F(d) \arrow[]{r}{\eta_d} \arrow[swap]{d}{F(f)} & G(d) \arrow[Rightarrow, from=dl, "\gamma^f"] \arrow[]{d}{G(f)} \\
    F(c) \arrow[swap]{r}{\eta_c} & G(c). \\
    \end{tikzcd}
    \]
    See
    \[
    \begin{tikzcd}
    \eta_c(x) \arrow{r}{\eta_c(\phi)} & \eta_c(F(f)(y)) \arrow{r}{\gamma^f_y} & G(f)(\eta_d(y)).\\
    \end{tikzcd}
    \]
    
    The reader will be able to check that $\eta^\sharp$ is indeed a functor. Also, it is clear that it is compatible with the projections $F^\sharp \longrightarrow \mathcal C$ and $G^\sharp \longrightarrow \mathcal C$, so that we can regard $\eta^\sharp$ as a $1$-cell in the slice $2$-category $\CAT/\mathcal C$.
\end{itemize}

Finally, suppose given a $2$-cell in $[\mathcal C^{op},\CAT]$, i.e. a modification $\mu:\eta \Rrightarrow \chi$ between pseudonatural transformations $\eta$, $\chi:F \implies G$. We define a natural transformation $\mu^\sharp:\eta^\sharp \implies \chi^\sharp$ as follows: for each $(c,x) \in Ob(F^\sharp)$, we take

$$
\mu^\sharp_{(c,x)}: \eta^\sharp((c,x))=(c,\eta_c(x)) \longrightarrow \chi^\sharp((c,x))=(c,\chi_c(x))
$$

to be $(id_c,\beta^c_{\chi_c(x)} \circ (\mu_c)_x)$, where $\beta^c$ is the natural isomorphism (associated to $G$ by the definition of a pseudofunctor) $1_{G(c)} \implies G(id_c)$. See

\[
\begin{tikzcd}
\eta_c(x) \arrow[]{r}{(\mu_c)_x} & \chi_c(x) \arrow[]{r}{\beta^c_{\chi_c(x)}} & G(id_c)(\chi_c(x)).
\end{tikzcd}
\]

The reader will be able to check that $\mu^\sharp$ is indeed a natural transformation. Furthermore, it can be verified that by sending a pseudofunctor $F$ to a category $F^\sharp$, a pseudonatural transformation $\eta:F \implies G$ to a functor $\eta^\sharp:F^\sharp \longrightarrow G^\sharp$, and a modification $\mu:\eta \Rrightarrow \chi$ to a natural transformation $\mu^\sharp:\eta^\sharp \implies \chi^\sharp$, we have defined a pseudofunctor

$$
-^\sharp:[\mathcal C^{op},\CAT] \longrightarrow \CAT/\mathcal C.
$$

\begin{definition}
(Grothendieck construction for covariant pseudofunctors)
\label{def:grotconstrcovariant}
\end{definition}
Let $\mathcal C$ be a $1$-category. Given some pseudofunctor $F:\mathcal C \longrightarrow \CAT$, we define its \textit{Grothendieck construction} or \textit{Grothendieck category}, denoted by $F_\sharp$, as the $1$-category given by the following data:

\begin{itemize}
    \item Its objects are pairs $(c,x)$, where $c \in Ob(\mathcal C)$ and $x \in Ob(F(c))$.
    
    \item An arrow $(c,x) \longrightarrow (d,y)$ is a pair $(f,\phi)$, where $f \in \mathcal C(c,d)$ and $\phi \in F(d)(Ff(x),y)$.
    
    \item The composite of morphisms $(f,\phi):(c,x)\longrightarrow (d,y)$ and $(g,\psi):(d,y)\longrightarrow (e,z)$ is defined as
    
    $$(g \circ f\text{  },\text{  } \psi \circ F(g)(\phi) \circ (\alpha^{f,g}_x)^{-1}),$$
    
    where $\alpha^{f,g}$ is the natural isomorphism (associated to $F$ by the definition of a pseudofunctor) $F(f) \circ F(g) \implies F(g \circ f)$. See
    
    \[
    \begin{tikzcd}
    F(g \circ f)(x) \arrow[]{r}{(\alpha^{f,g}_x)^{-1}}
    & (Fg \circ Ff)(x)=Fg(Ff(x)) \arrow[]{r}{Fg(\phi)}
    & Fg(y) \arrow[]{r}{\psi}
    & z.
    \end{tikzcd}
    \]
    
\end{itemize}

The reader will be able to check that composition is associative and that each object possesses an identity arrow (by using the natural isomorphisms $\alpha^c:1_{F(c)}\implies F(id_c)$). As in the previous definition, $F_\sharp$ has a canonical projection functor $F_\sharp \longrightarrow \mathcal C$ given by $(c,x) \longmapsto c$ and $(f,\phi) \longmapsto f$. (Here, the reader might recognize it as what is called in the literature an \textit{opfibration}, or that it realizes $F_\sharp$ as an \textit{opfibered category} over $\mathcal C$).

Suppose given a $1$-cell in $[\mathcal C, \CAT]$, i.e. a pseudonatural transformation $\eta:F \implies G$. We define a functor $\eta_\sharp:F_\sharp \longrightarrow G_\sharp$ as follows:

\begin{itemize}
    \item $\eta_\sharp((c,x))=(c,\eta_c(x))$ for each $(c,x) \in Ob(F_\sharp)$.
    
    \item For each $(f,\phi):(c,x) \longrightarrow (d,y)$ in $F_\sharp$, we define $\eta_\sharp((f,\phi)):(c,\eta_c(x)) \longrightarrow (d,\eta_d(y))$ as
    
    $$(f \text{  },\text{  } \eta_d(\phi) \circ (\gamma^f_x)^{-1}),$$
    
    where $\gamma^f$ is the natural isomorphism (associated to $\eta$ by the definition of a pseudonatural transformation) as in
    
    \[
    \begin{tikzcd}
    F(d) \arrow[]{r}{\eta_d} \arrow[swap,leftarrow]{d}{F(f)} & G(d)  \arrow[leftarrow]{d}{G(f)} \\
    F(c) \arrow[swap]{r}{\eta_c} & G(c). \arrow[Rightarrow, from=ul, "\gamma^f"] \\
    \end{tikzcd}
    \]
    See
    \[
    \begin{tikzcd}
    G(f)(\eta_c(x)) \arrow[]{r}{(\gamma^f_x)^{-1}} & \eta_d(F(f)(x)) \arrow[]{r}{\eta_d(\phi)} & \eta_d(y).
    \end{tikzcd}
    \]

    The reader will be able to check that $\eta_\sharp$ is indeed a functor. Again, it is clearly compatible with the projections $F_\sharp \longrightarrow \mathcal C$ and $G_\sharp \longrightarrow \mathcal C$, so that we can regard $\eta_\sharp$ as a $1$-cell in the slice $2$-category $\CAT/\mathcal C$.
\end{itemize}

Suppose given a $2$-cell in $[\mathcal C,\CAT]$, i.e. a modification $\mu:\eta \Rrightarrow \chi$ between pseudonatural transformations $\eta$, $\chi:F \implies G$. We define a natural transformation $\mu_\sharp:\eta_\sharp \implies \chi_\sharp$ as follows: for each $(c,x) \in Ob(F_\sharp)$, we take

$$
(\mu_\sharp)_{(c,x)}: \eta_\sharp((c,x))=(c,\eta_c(x)) \longrightarrow \chi_\sharp((c,x))=(c,\chi_c(x))
$$

to be $(id_c, (\mu_c)_x \circ (\beta^c_{\eta_c(x)})^{-1})$, where $\beta^c$ is the natural isomorphism (associated to $G$ by the definition of a pseudofunctor) $1_{G(c)} \implies G(id_c)$. See

\[
\begin{tikzcd}
G(id_c)(\eta_c(x)) \arrow[]{r}{(\beta^c_{\eta_c(x)})^{-1}}
& \eta_c(x) \arrow[]{r}{(\mu_c)_x} & \chi_c(x).
\end{tikzcd}
\]

The reader will be able to check that $\mu_\sharp$ is indeed a natural transformation. As before, it can be verified that by sending a pseudofunctor $F$ to $F_\sharp$, a pseudonatural transformation $\eta:F \implies G$ to $\eta_\sharp:F_\sharp \longrightarrow G_\sharp$, and a modification $\mu:\eta \Rrightarrow \chi$ to $\mu_\sharp:\eta_\sharp \implies \chi_\sharp$, we have defined a pseudofunctor

$$
-_\sharp:[\mathcal C,\CAT] \longrightarrow \CAT/\mathcal C.
$$

\subsubsection{Representable pseudofunctors}

Let $\textbf C$ be a $2$-category. For each $c \in Ob(\textbf C)$, we define a pseudofunctor (in fact, a strict $2$-functor) $\textbf C(-,c): \textbf C^{op} \longrightarrow \CAT$ as follows:

\begin{itemize}
    \item Each $d \in Ob(\textbf C)$ is sent to the hom-category $\textbf C(d,c)$.
    
    \item Each $1$-cell $f:d \longrightarrow e$ in $\textbf C$ is sent to the functor $\textbf C(f,c):\textbf C(e,c) \longrightarrow \textbf C(d,c)$ given by precomposition of both $1$-cells and $2$-cells with $f$.
    
    \item Each $2$-cell $\eta:f \implies g$ between $1$-cells $f$, $g:d \longrightarrow e$ is sent to the natural transformation
    
    $$
    \textbf C(\eta,c):\textbf C(f,c) \implies \textbf C(g,c)
    $$
    
    given by precomposition with $\eta$, that is, by associating to each $1$-cell $h:e \longrightarrow c$ (i.e. object of $\textbf C(e,c)$) the $2$-cell (i.e. morphism of $\textbf C(d,c)$)
    
    $$
    \textbf C(\eta,c)_h=h \circ \eta: h \circ f \longrightarrow h \circ g.
    $$
    \end{itemize}

Next, given a $1$-cell $p:c \longrightarrow c'$ in $\textbf C$, we define a pseudonatural transformation (in fact, a strict $2$-natural transformation) $\textbf C(-,p):\textbf C(-,c) \implies \textbf C(-,c')$ as follows:

\begin{itemize}
    \item To each $d \in Ob(\textbf C)$ we associate the functor (i.e. $1$-cell in $\CAT$) $\textbf C(d,p):\textbf C(d,c) \longrightarrow \textbf C(d,c')$ given by postcomposition of both $1$-cells and $2$-cells with $f$.
    
    \item As we are only dealing with strict $2$-categories, composition of $1$-cells in $\textbf C$ is strictly associative, hence we can fill the square diagrams thus obtained with identity natural transformations.
\end{itemize}

Given a $2$-cell $\eta:p \implies p'$ between $p$, $p':c \longrightarrow c'$, we define a modification $\textbf C(-,\eta):\textbf C(-,p) \Rrightarrow \textbf C(-,p')$ by associating to each $d \in Ob(\textbf C)$ the natural transformation $\textbf C(d,\eta):\textbf C(d,p) \implies \textbf C(d,p')$ given on each $f \in Ob(\textbf C(d,c))$ by $\textbf C(d,\eta)_f=\eta \circ f:p \circ f \longrightarrow p' \circ f$.

Routine diagram chasing shows that the above constructions define a strict $2$-functor $\textbf C \longrightarrow [\textbf C^{op},\CAT]$, which we denote by $\mathcal Y_\textbf C$ and call the \textit{Yoneda embedding} associated to $\textbf C$.

\begin{Obs}
The above constructions can be adapted to produce a Yoneda embedding for any \textit{weak} $2$-category $\textbf C$. In this case, $\mathcal Y_\textbf C$ will in general only be a (non-strict) pseudofunctor. Also, the term \textit{embedding} used here may be misleading in that the $2$-categorical statement analogous to the Yoneda lemma, although true, is not nearly immediate from the above discussion. An elementary but not-so-short proof is given in \cite{Bak1}.
\end{Obs}

\subsubsection{Adjunctions in a $2$-category}

\begin{definition}
Let $\textbf C$ be a $2$-category. An \textit{adjunction} in $\textbf C$ is a quadruple $(f,g,\eta,\varepsilon)$, where:

\begin{itemize}
    \item $f$ and $g$ are $1$-cells in $\textbf C$ of the form $f:c \longrightarrow d$, $g:d \longrightarrow c$.
    
    \item $\eta$ and $\varepsilon$ are $2$-cells of the form $\eta:id_c \implies g \circ f$, $\varepsilon:f \circ g \implies id_d$.
    
    \item These satisfy the identities $(\varepsilon f) \circ (f \eta) = 1_f$ and $(g \varepsilon)\circ(\eta g)=1_g$.
\end{itemize}

We denote the existence of such an adjunction by $f \dashv g$.
\end{definition}

For our purposes, the crucial property of adjunctions in $2$-categories is that they are (up to isomorphism) preserved by any pseudofunctor:

\begin{lemma}
\label{lem:adjunctionsarepreservedbypseudofunctors}
Let $F:\textbf C \longrightarrow \textbf D$ be a pseudofunctor, and $(f,g,\eta,\varepsilon)$ an adjunction in $\textbf C$. Then $F$ induces an adjunction $(F(f),F(g),\bar{\eta},\bar{\varepsilon})$ in $\textbf D$.
\end{lemma}

\begin{proof}
Let $f:c \longrightarrow d$, $g:d \longrightarrow c$. Take $\bar{\eta}:id_{F(c)} \implies F(g) \circ F(f)$ to be the composite

$$
id_{F(c)} \overset{\alpha^c}{\implies} F(id_c) \overset{F(\eta)}{\implies} F(g \circ f) \overset{(\alpha^{g,f})^{-1}}{\implies} F(g) \circ F(f),
$$

where $\alpha^c$ and $\alpha^{g,f}$ are the $2$-cells associated to $F$ as a pseudofunctor. Analogously, take $\bar{\varepsilon}:F(f) \circ F(g) \implies id_{F(d)}$ to be the composite

$$
F(f) \circ F(g) \overset{\alpha^{f,g}}{\implies} F(f \circ g) \overset{F(\varepsilon)}{\implies} F(id_d) \overset{(\alpha^d)^{-1}}{\implies} id_{F(d)}.
$$

Now, notice that

\begin{align*}
    (\bar{\varepsilon}F(f))\circ (F(f)\circ \bar{\eta}) & = (((\alpha^d)^{-1}F(\varepsilon)\alpha^{f,g})F(f))\circ (F(f)( (\alpha^{g,f})^{-1}F(\eta) \alpha^c ))
\end{align*}

is given by the following composite of $2$-cells:

$$
F(f) \overset{F(f)\alpha^c}{\implies} F(f)\circ F(id_c) \overset{F(f)F(\eta)}{\implies} F(f)\circ F(g \circ f) \overset{F(f)(\alpha^{g,f})^{-1}}{\implies} F(f) \circ F(g) \circ F(f) \implies
$$

$$
\overset{\alpha^{f,g}F(f)}{\implies} F(f \circ g) \circ F(f) \overset{F(\varepsilon)F(f)}{\implies} F(id_d) \circ F(f) \overset{(\alpha^d)^{-1}F(f)}{\implies} F(f).
$$

On the other hand, the equality $(\varepsilon f) \circ (f \eta) = 1_f$ implies (by functoriality of $\textbf C(c,d) \longrightarrow \textbf D(F(c),F(d))$) $F(\varepsilon f) \circ F(f \eta) = 1_{F(f)}$. The left-hand side equals the composite of $2$-cells

$$
F(f) \overset{F(f\eta)}{\implies} F(f \circ g \circ f) \overset{F(\varepsilon f)}{\implies} F(f),
$$

which (by expanding $id_{F(f \circ g \circ f)}$ through the coherence laws of $F$ as a pseudofunctor) can be rewritten as

$$
F(f) \overset{F(f\eta)}{\implies} F(f \circ g \circ f) \overset{(\alpha^{f,g \circ f})^{-1}}{\implies} F(f) \circ F(g \circ f) \overset{F(f)(\alpha^{g,f})^{-1}}{\implies} F(f) \circ F(g) \circ F(f) \implies
$$

$$
\overset{\alpha^{f,g}F(f)}{\implies} F(f \circ g) \circ F(f) \overset{\alpha^{f \circ g,f}}{\implies} F(f \circ g \circ f) \overset{F(\varepsilon f)}{\implies} F(f).
$$

Again by using the coherence laws of $F$, it can be shown (as the reader will be able to do in detail) that the following equalities hold:

$$
(F(f)F(\eta))\circ (F(f)\alpha^c) = (\alpha^{f,g \circ f})^{-1} \circ F(f\eta):F(f) \implies F(f) \circ F(g \circ f),
$$

$$
((\alpha^d)^{-1}F(f)) \circ (F(\varepsilon)F(f)) =
F(\varepsilon f) \circ \alpha^{f \circ g,f}: F(f \circ g) \circ F(f) \implies F(f).
$$

It follows that the two composites of $2$-cells above are equal, so that $(\bar{\varepsilon}F(f))\circ (F(f)\circ \bar{\eta})=1_{F(f)}$, which is the first desired identity. The second one can be shown analogously.
\end{proof}

\subsection{Generalized categories of institution-like structures}\label{Categories of institutions as Grothendieck categories}

\cite{Diac2} describes a procedure to recover $\inm$ as a Grothendieck category. It is done by introducing the so-called \textit{category of rooms}, denoted by $\Room$ (see below), so that $\inm$ is canonically equivalent (isomorphic, in fact) to $\CAT((-)^{op},\Room)^\sharp$. Before recalling this construction, it will be convenient to define (or better, to fix notation for) a general notion of $\Room$-like category which can be applied to produce other categories of institution-like objects.

\begin{definition}
\label{def:roomcategory}
\end{definition}
Let $\mathcal C$ be a $1$-category. We say that a $1$-category $R$ is a \textit{category of rooms} for $\mathcal C$ if there exists an equivalence of categories $\mathcal C \simeq \CAT(-^{op},R)^\sharp$, where the right-hand side denotes the category obtained as in

\[
\begin{tikzcd}[row sep=tiny]
    \CAT & \left[ \CAT^{op},\CAT'\right] & \left[ \cat^{op},\CAT' \right] & \CAT'/\cat \\
    \vertin & \vertin & \vertin & \vertin \\
    R \arrow[mapsto]{r}{} & \CAT(-^{op},R) \arrow[mapsto]{r}{} & \CAT(-^{op},R) \arrow[mapsto]{r}{} & \CAT(-^{op},R),^\sharp \\
\end{tikzcd}
\]

where we denote by $\CAT'$ a $2$-category of categories defined in a Grothendieck universe larger than that of $\CAT$. As discussed in the previous subsection, both the Yoneda embedding for $2$-categories and the Grothendieck construction are pseudofunctorial. It is then immediate that the above construction gives rise to a pseudofunctor (in fact, a strict $2$-functor)

\begin{align*}
    \CAT & \longrightarrow \CAT'/\cat\\
    R & \longmapsto \CAT(-^{op},R)^\sharp.
\end{align*}

It will be denoted by $\textbf{ins}$ and called \textit{institutional realization}.

It often happens that the right Grothendieck construction to be used is that from Definition \ref{def:grotconstrcovariant}, for covariant pseudofunctors. We say that $R$ is a \textit{category of co-rooms} for $\mathcal C$ if there exists an equivalence of categories $\mathcal C \simeq (\CAT(-^{op},R)_\sharp)^{op}$. See

\[
\begin{tikzcd}[column sep=tiny, row sep=tiny]
    \CAT & \left[ \CAT^{op},\CAT'\right] & \left[ \cat^{op},\CAT' \right] & \CAT'/\cat^{op} & \CAT'^{co}/\cat \\
    \vertin & \vertin & \vertin & \vertin & \vertin \\
    R \arrow[mapsto]{r}{} & \CAT(-^{op},R) \arrow[mapsto]{r}{} & \CAT(-^{op},R) \arrow[mapsto]{r}{} & \CAT(-^{op},R)_\sharp \arrow[mapsto]{r}{} & (\CAT(-^{op},R)_\sharp).^{op}
\end{tikzcd}
\]

Once again, we obtain a pseudofunctor (in fact, a strict $2$-functor)

\begin{align*}
    \CAT & \longrightarrow \CAT'^{co}/\cat\\
    R & \longmapsto (\CAT(-^{op},R)_\sharp)^{op},
\end{align*}

which we denote by $\textbf{coins}$ and call \textit{institutional co-realization}.

\begin{Obs}
It is clear that $\CAT$ plays no distinguished role in this construction besides being a $2$-category. The inner $op$ as in $\CAT(-^{op},R)$ and $(\CAT(-^{op},R)_\sharp)^{op}$ corresponds (see Example \ref{example room tradicional}) to the fact that we wish the functors sending signatures to categories of models to be contravariant. The outer $op$ as in $(\CAT(-^{op},R)_\sharp)^{op}$ (as well as its absence from $\CAT(-^{op},R)$) corresponds to the fact that we wish any morphism between institution-like objects to have the same direction as its corresponding functor between signature categories. The $co$ as in $\CAT'^{co}/\cat$ is due to the fact that the pseudofunctor taking a category to its opposite reverses the direction of natural transformations, but not of functors. Since left-right adjunctions in $\CAT'$ correspond to right-left adjunctions in $\CAT'^{co}$, Lemma \ref{lem:adjunctionsarepreservedbypseudofunctors} implies that $\textbf{coins}$ sends left-right adjunctions in $\CAT$ to right-left adjunctions in $\CAT'^{co}/\cat$.
\end{Obs}

We list below some examples of room categories for some categories of institution-like objects. Proofs will not be given, but the reader will be able to provide them without difficulty.

\begin{example}
\label{example room tradicional}
($\Room$, a room category for $\inm$ and $\inc$)
\end{example}

Define a category $\Room$ as follows:

\begin{itemize}
    \item Its objects are triples $\langle S,M,(R_m)_{m\in Ob(M)}\rangle$, where $S$ is a set, $M$ is a category, and, for each $m\in Ob(M)$, $R_m:S\to 2=\{0,1\}$ is a function.
    
    \item A morphism $\langle S,M,(R_m)_{m\in Ob(M)} \rangle \xrightarrow{(\sigma,\mu)}\langle S',M',(R'_{m'})_{m'\in Ob(M')}\rangle$ consists of a function $\sigma:S'\to S$ and a functor $\mu:M\to M'$ such that $R'_{\mu m}(s) = R_m\sigma(s)$ for every $m\in Ob(M)$ and $s\in Ob(S)$.
    
    \item Composition is given by $(\sigma',\mu')\circ (\sigma,\mu)=(\sigma \circ \sigma',\mu' \circ \mu)$.
\end{itemize}

It is clear that $\Room$ is indeed a category. Then, in the terminology introduced above, we have

$$
\inm \cong \textbf{ins}(\Room),
$$

$$
\inc \cong \textbf{coins}(\Room).
$$

Both projections $\textbf{ins}(\Room) \longrightarrow \cat$ and $\textbf{coins}(\Room) \longrightarrow \cat$ recover the underlying category of signatures of an institution. For more on this example, we refer the reader to \cite{Diac2}.

\begin{example}
($\pi\Room$, a room category for $\pi\mathbf{Ins}_{mor}$ and $\pi\mathbf{Ins}_{co}$)
\end{example}

Define a category $\pi\Room$ as follows:

\begin{itemize}
    \item Its objects are pairs $\langle S, C \rangle$, where $S$ is a set and $C:2^S \longrightarrow 2^S$ is a closure operator (we give $2^S \cong \mathscr P(S)$ the canonical ordering).
    
    \item A morphism $\langle S,C \rangle \overset{\sigma}{\longrightarrow} \langle S',C' \rangle$ consists of a function $\sigma:S' \longrightarrow S$ such that $\sigma^* \circ C = C' \circ \sigma^*$, where $\sigma^*:2^S \longrightarrow 2^{S'}$ is the function given by pulling back along $\sigma$ (or by taking preimages).
    
    \item Composition is given by $\sigma' \circ_{\pi\Room} \sigma = \sigma \circ_{Set} \sigma'$.
    \end{itemize}

It is clear that $\pi\Room$ is indeed a category. It is easily shown that

$$
\pi\mathbf{Ins}_{mor} \cong \textbf{ins}(\pi\Room),
$$

$$
\pi\mathbf{Ins}_{co} \cong \textbf{coins}(\pi\Room).
$$

Both projections $\textbf{ins}(\pi\Room) \longrightarrow \cat$ and $\textbf{coins}(\pi\Room) \longrightarrow \cat$ recover the underlying category of signatures of a $\pi$-institution.

\begin{example}
(The terminal category, a room category for $\cat$)
\end{example}

Let $1=\{*\}$ denote the terminal category. It is immediate that both $\textbf{ins}(1)$ and $\textbf{coins}(1)$ are canonically isomorphic to $\cat$ via the projections provided by the Grothendieck construction.

\begin{example}
(Institution-like structures versus diagrams)
\end{example}

\textbf{ins} and \textbf{coins} are essentially the same, respectively, as the constructions of categories of diagrams $Diac_{mor}$ and $Diag_{co}$ given (in an ad hoc way) in Section 2. Indeed, for any category $\mathcal C$ there are canonical isomorphisms of categories
$$
\textbf{ins}(\mathcal C) \cong Diag_{mor}(\mathcal C^{op}),
$$
$$
\textbf{coins}(\mathcal C) \cong Diag_{co}(\mathcal C^{op}),
$$
both given on objects by sending a pair $(\mathcal A, F:\mathcal A \rightarrow \mathcal C)$ to $(\mathcal A, F^{op}:\mathcal A \rightarrow \mathcal C^{op})$.

Moreover, for each $\mathcal C$ we have an isomorphism
$$
Diag_{mor}(\mathcal C) \cong Diag_{co}(\mathcal C^{op})
$$
also given by sending a pair $(\mathcal A, F:\mathcal A \rightarrow \mathcal C)$ to $(\mathcal A, F^{op}:\mathcal A \rightarrow \mathcal C^{op})$. It then follows that for each $\mathcal C$ we have a sequence of isomorphisms
$$
\textbf{ins}(\mathcal C) \cong Diag_{mor}(\mathcal C^{op}) \cong Diag_{co}(\mathcal C) \cong \textbf{coins}(\mathcal C^{op}).
$$
An immediate corollary of this is:
\begin{itemize}
    \item $\Room^{op}$ (resp. $\pi\Room^{op}$) is a category of rooms for $\inc$ (resp. $\pi\inc$).
    \item $\Room^{op}$ (resp. $\pi\Room^{op}$) is a category of co-rooms for $\inm$ (resp. $\pi\inm$).
\end{itemize}

Although the constructions of categories of diagrams and of institutional realizations are equally expressive, \textbf{ins} and \textbf{coins} fit better into the institutional framework, while $Diag_{mor}$ and $Diac_{co}$ would be more natural from a general categorical point of view.

\subsection{Recovering adjunctions between categories of ($\pi$-)institutions}

Lemma \ref{lem:adjunctionsarepreservedbypseudofunctors} ensures us that $\textbf{ins}$ preserves adjunctions, and that $\textbf{coins}$ reverses adjunctions. As a result, the adjunctions between categories of institution-like objects described in the previous sections can be given a simple and uniform treatment as images under $\textbf{ins}$ or $\textbf{coins}$ of certain adjunctions between the room categories attributed to them in the previous subsection.

\begin{example}
($\inm$ and $\pi\mathbf{Ins}_{mor}$)
\end{example}

Define functors $\mathscr F:\Room \longrightarrow \pi\Room$ and $\mathscr G:\pi\Room \longrightarrow \Room$ as follows:

\begin{itemize}
    \item For each object $r=\langle S, M, (R_m)_{m \in Ob(M)} \rangle$ of $\Room$, we define $\mathscr F(r)$ as $\langle S, C^r \rangle$, where $C^r:\mathscr P(S) \longrightarrow \mathscr P(S)$ is given by sending each $S' \subset S$ to
    
    $$
    \{s \in S \text{ such that } R_m(s)=1 \text{ for every } m \in Ob(M) \text{ such that } R_m(S')=\{1\} \}.
    $$
    
    A morphism $\langle S,M,(R_m)_{m\in Ob(M)} \rangle \xrightarrow{(\sigma,\mu)}\langle S',M',(R'_{m'})_{m'\in Ob(M')}\rangle$ is sent to $\sigma$.
    
    \item For each object $r\!=\!\langle S, C \rangle$ of $\pi\Room$, we define $\mathscr G(r)$ as $\langle S, \mathscr P(S), (\chi_m)_{m \in Ob(\mathscr P(S))} \rangle$, where $\mathscr P(S)$ is given the structure of a co-discrete category, and for each $m \subset S$, $\chi_m:S \longrightarrow 2$ is the characteristic function of $m$.
    
    A morphism $\langle S, C \rangle \overset{\sigma}{\longrightarrow} \langle S', C' \rangle$ is sent to $(\sigma,\sigma^*)$, where $\sigma^*:\mathscr P(S) \longrightarrow \mathscr P(S')$ is the functor between co-discrete categories given on objects by taking preimages.
\end{itemize}

One can then easily describe an adjunction $\mathscr G \dashv \mathscr F$ and show that $\mathscr G$ is fully faithful (hence it realizes $\pi\Room$ as a coreflective subcategory of $\Room$). It follows from Lemma \ref{lem:adjunctionsarepreservedbypseudofunctors}, and from the fact that pseudofunctors preserve isomorphisms between $1$-cells, that the functors

$$
\textbf{ins}(\mathscr F):\textbf{ins}(\Room) \cong \inm \longrightarrow \textbf{ins}(\pi\Room) \cong \pi\mathbf{Ins}_{mor},
$$
$$
\textbf{ins}(\mathscr G):\textbf{ins}(\pi\Room) \cong \pi\mathbf{Ins}_{mor} \longrightarrow \textbf{ins}(\Room) \cong \inm
$$

satisfy $\textbf{ins}(\mathscr G) \dashv \textbf{ins}(\mathscr F)$, and that $\textbf{ins}(\mathscr G)$ realizes $\textbf{ins}(\pi\Room)$ (resp. $\pi\mathbf{Ins}_{mor}$) as a coreflective subcategory of $\textbf{ins}(\Room)$ (resp. $\inm$).

\begin{example}
($\inc$ and $\pi\mathbf{Ins}_{co}$)
\end{example}

Let $\mathscr F$ and $\mathscr G$ be as in the previous example. The same argument shows that the functors

$$
\textbf{coins}(\mathscr F):\textbf{coins}(\Room) \cong \inc \longrightarrow \textbf{coins}(\pi\Room) \cong \pi\mathbf{Ins}_{co},
$$
$$
\textbf{coins}(\mathscr G):\textbf{coins}(\pi\Room) \cong \pi\mathbf{Ins}_{co} \longrightarrow \textbf{coins}(\Room) \cong \inc
$$

satisfy $\textbf{coins}(\mathscr F) \dashv \textbf{coins}(\mathscr G)$, and that $\textbf{coins}(\mathscr G)$ realizes $\textbf{coins}(\pi\Room)$ (resp. $\pi\mathbf{Ins}_{co}$) as a reflective subcategory of $\textbf{coins}(\Room)$ (resp. $\inc$).

\begin{example}
(Categories of ($\pi$-)institutions and $\cat$)
\end{example}

We leave to the reader the exercise of defining adjoints (left, right, or both) to the terminal functors $\Room \to 1$ and $\pi\Room \to 1$ using the methods described here, in order to produce several canonical adjunctions between $\cat$ and categories of ($\pi$-)institutions.

\begin{example}
(Categories of ($\pi$-)institutions and categories of diagrams)
\end{example}

Any adjunction of the form
\[\begin{tikzcd}
	{\Room} && {\mathcal C^{op}}
	\arrow["{R}"{name=0}, from=1-3, to=1-1, bend left]
	\arrow["{L}"{name=1}, from=1-1, to=1-3, bend left]
	\arrow["\dashv"{rotate=-90}, from=1, to=0, phantom]
\end{tikzcd}
\]
    induces two adjunctions: one between $\inm = \textbf{ins}(\Room)$ and $\textbf{ins}(\mathcal C^{op}) \cong Diag_{mor}(\mathcal C)$, and one between $\inc = \textbf{coins}(\Room)$ and $\textbf{coins}(\mathcal C^{op}) \cong Diag_{co}(\mathcal C)$. Analogously, an adjunction of the form
\[\begin{tikzcd}
	{\Room} && {\mathcal C}
	\arrow["{R}"{name=0}, from=1-3, to=1-1, bend left]
	\arrow["{L}"{name=1}, from=1-1, to=1-3, bend left]
	\arrow["\dashv"{rotate=-90}, from=1, to=0, phantom]
\end{tikzcd}
\]
induces an adjunction between $\inm = \textbf{ins}(\Room)$ and $\textbf{ins}(\mathcal C) \cong Diag_{co}(\mathcal C)$, and another one between $\inc = \textbf{coins}(\Room)$ and $\textbf{coins}(\mathcal C) \cong Diag_{mor}(\mathcal C)$. Analogously for $\pi\Room$ (or any category whatsoever) in place of $\Room$, and for $R \dashv L$ in place of $L \dashv R$.



{\color{red}{}}

\section{Propositional logics and ($\pi$-)institutions}


In this section, we present several different ways of connecting abstract propositional logics to institutions and $\pi$-institutions. 


In subsection 4.1 we have described the  $\pi$-institutions associated to  categories of abstract propositional logics and some forms of translation morphisms,  as developed in \cite{MaPi1}. This naturally  lead us to search  an analogous ``model-theoretical" version of it that is different from the canonical one i.e., that  obtained by applying the functor $G : \pic \rightarrow \inc$ (see subsections 1.3 and 2.1). This is achieved in section 4.2, based on the development made in the section 3.1 of \cite{MaPi3}: we   provide (another) \emph{institutions} for each   category of propositional logics, through the use of the notion of a {\em matrix} for a propositional logic. It should be mentioned that the use of institutional-theoretic devices are useful for establishing an abstract Glivenko's theorem for algebraizable logics regardless of their particular signatures associated (see \cite{MaPi3}).

In \cite{AMP1} was introduced the concept of (finitary) filter pair, that can be seem as a categorial presentation of a propositional logic, in fact the category of logics is isomorphic to a coreflective subcategory of the category of filter pairs. In the subsection \ref{FilterPairInst} we present a functor $\mathcal{F}i \to \inm$, from the category of filter pairs, $\mathcal{F}_i$,
to the category of all institutions and morphisms, $\inm$. This is qualitatively  different connection from  the obtained in subsections \ref{pi-inst-examples} and \ref{instprop} between propositional logic and ($\pi$-)institution. From the adjunctions between the categories of logics and of filter pairs, ${\cal{L}} \leftrightarrows {\cal{F}}i$, and the adjunction between the categories of institutions and of $\pi$-institutions, $\pi-\inm \leftrightarrows \inm$, we obtain directly functors: $\mathcal{F}i \to \pi-\inm$, ${\cal L} \to \inm$, ${\cal L} \to \pi-\inm$. We finish this section with some remarks, indicating some generalizations concerning the use of multialgebras (a concept that will appear again in Section 5) in the setting of abstract propositional logic, including a natural generalization of the notion of filter pairs.



   

\subsection{A $\pi$-institution for the abstract propositional logics}
\label{pi-inst-examples} 

Here we describe the $\pi$-institutions associated to categories of abstract propositional logics and some forms of translation morphisms,  as developed in \cite{MaPi1}. 

In \cite{AFLM}, \cite{FC} and \cite{MaMe} are considered some categories of propositional logics, namely ${\cal L}_{s}$ and ${\cal L}_{f}$, where:
\begin{itemize}
\item the objects are of the form $l = (\Sigma, \vdash)$,  where $\Sigma = (\Sigma_n)_{n \in \N}$  is  finitary signature, $Form(\Sigma) = Fm_\Sigma(X)$ is the absolutely free $\Sigma$-algebra of formulas on a fixed enumerable set of variables $X$  and $\vdash \subseteq P(Form(\Sigma))\times Form(\Sigma)$ is a tarskian consequence operator;
\item the morphisms $f : (\Sigma, \vdash) \rightarrow (\Sigma', \vdash')$ are of the form $f : \Sigma \rightarrow \Sigma'$ with the former category  having ``strict" ($n$-ary symbol to $n$-ary symbol) morphisms and the latter ``flexible" ($n$-ary symbol to $n$-ary term) morphisms.
\end{itemize}
To the category ${\cal L}_f$ is associated an $\pi$-institution  $J_f$ in the following way:
\begin{itemize}
    \item $\sig_f \coloneqq  {\cal L}_f$;
    \item $Sen_f : \sig_f \to \set$ is given by $(g: (\Sigma, \vdash) \to (\Sigma', \vdash)) \ \mapsto\ (\hat{g} : Form(\Sigma) \to Form(\Sigma'))$, where $\hat{g}$ is the usual expansion to formulas;
    \item  For each $l = (\Sigma, \vdash) \in |\sig_f|$ and $\Gamma \sub  Form(\Sigma)$, we define $C_{l}(\Gamma)\coloneqq  \{\phi \in  Form(\Sigma) :\Gamma\vdash_l \phi\}$.\\
    
    An analogous process is used to form $J_s$ from ${\cal L}_s$.
\end{itemize}

In \cite{MaMe}, the ``inclusion" functor $(+)_L : {\cal L}_s \to {\cal L}_f$ induces a comorphism (and also a morphism) on the associated $\pi$-institutions $(+) := ((+)_L, \alpha^+) : J_s \to J_f$, where, for each $l = (\Sigma, \vdash) \in \sig_s = {\cal L}_s$, $\alpha^+(l) = Id_{Form(\Sigma)} : Form(\Sigma) \to Form(\Sigma)$. The paper also presents a right adjoint $(-)_L : {\cal L}_f \to {\cal L}_s$ to the ``inclusion" functor. Essentially this fuctor sends a signature $\Sigma$ to its derived one $(-)_{L}\Sigma:=(Form(\Sigma)[n])_{n\in\cal{N}}$. We have also a comorphism of $\pi$-institutions  associated to this functor. Notice that given some logic $l=(\Sigma,\vdash)$, we have $Sen_{s}(-)_{L}(l)=Form((-)_{L}\Sigma)=Form(\Sigma)$. So the fuctor $(-)_{L}$ induces a comorphism $((-)_{L},\alpha^-)$ where $\alpha^-$ is the identity between formulas. It will be interesting  understand the role of these adjoint pair of functors  between the logical categories (${\cal L}_f, {\cal L}_s$) at the $\pi$-institutional level ($J_f, J_s$).


\vtres


\subsection{An institution for the abstract propositional logics}\label{instprop}

We now present an alternative institutionalization of propositional logic. This assignment is used in \cite{MaPi3} to establish an abstract Glivenko's theorem for algebraizable logics.

Let $l=(\Sigma,\vdash)$ be a logic and $M\in \Sigma-Str$. A subset $F$ of $M$ is a $l$-filter is for every $\Gamma\cup\{\varphi\}\subseteq Form(\Sigma)$ such that $\Gamma\vdash \varphi$ and every valuation $v:Form(\Sigma)\to A$, if $v[\Gamma]\subseteq F$ then $v(\varphi)\in F$. The pair $\langle M,F\rangle$ is then said to be a matrix model of $l$. The class of all matrix model of $l$ is denoted by $Matr_{l}$.This class is the class of objects of a category, also denoted by $Matr_l$: a morphism $h : \langle M, F \rangle \to \langle M', F' \rangle$ is a $\Sigma$-homomorphism $h : M \to M'$ such that $h^{-1}[F'] = F$; composition and identities are inherited from $\Sigma-Str$.

From  to the category of logics ${\cal L}_f$ (also to ${\cal L}_{s}$), we define:
\begin{itemize}
    \item $\sig_{} \coloneqq  {\cal L}_f$, the category of propositional logics $l=(\Sigma,\vdash)$ and flexible morphisms.
    \item $Sen_{}:\sig_{}\to\set$ where $Sen_{}(l)=\mathcal{P}(Form(\Sigma))\times Form(\Sigma)$ and given $f\in Mor_{\sig}(l_{1},l_{2})$ then $Sen_{}(f):Sen_{}(l_{1})\to Sen_{}(l_{2})$ is such that $Sen_{}(f)(\langle\Gamma,\varphi\rangle)=\langle f[\Gamma],f(\varphi)\rangle$. It is easy to see that $Sen_{}$ is a functor.
    \item $Mod_{}:\sig_{}\to\cat^{op}$ where $Mod_{}(l)=Matr_{l}$ and given $f\in Mor_{\sig}(l_{1},l_{2})$, $Mod_{}(f):Matr_{l_{2}}\to Matr_{l_{1}}$ such that $Mod_{}(f)(\langle M',F'\rangle)=\langle f^{\star}(M'),F'\rangle$. Here $f^{\star}:\Sigma'\!-\!str\to \Sigma\!-\!str$ is a functor that ``commutes over $Set$'' induced by the morphism $f$ where the interpretation of connectives are: $c_{n}^{f^{\star}M'}\coloneqq f(c_{n})^{M'}$ for all $c_{n}\in\Sigma$ (more details in \cite{MaPi3}). 
    \item Given $l = (\Sigma, \vdash)\in|\sig_{}|$, $\langle M,F\rangle\in |Mod_{}(l)|$ and $\langle\Gamma,\varphi\rangle\in Sen_{}(l)$ define the relation $\models_l\subseteq|Mod_{}(l)|\times Sen(l)$ as:
    
    \[\langle M,F\rangle\models_{l}\langle\Gamma,\varphi\rangle\ if\!f\ for\ all\ \, v\!:\!Form(\Sigma)\to M,\ if\ v[\Gamma]\subseteq F,\ then\ v(\varphi)\in F.\]
\end{itemize}

In \cite{MaPi3}, section 3.1, it is proven that this construction defines indeed an institution.

It should be noted that this institution and the $\pi$-institution described in the previous subsection, shares the same $\sig_{}$ ($ = {\cal L}_f$), but {\em are not} connected by  the canonical relation (adjunction) between institutions and $\pi$-institutions.

\subsection{Filter pairs as institutions}\label{FilterPairInst}

The notion of (finitary) filter pair, introduced in  \cite{AMP1}, can be seem as  a categorical presentation of a propositional logic. Here we recall the precise definition of this notion   and associate an institution to the category of all filter  pairs.

\begin{Df}\label{DefinitionFilterPairGen}
Let $\Sigma$ be a signature. A {\bf Filter Pair} over $\Sigma$ is a pair $(F,i)$, consisting of a contravariant functor $F:\Sigma\!-\!\text{str}\,^{op}\to \mathbf{CLat}$, from $\Sigma$-structures to complete lattices, and a collection of maps $i=(i_{M}:F(M)\to (\mathcal{P}(M),\subseteq))_{M\in \Sigma\!-\!\text{str}}$ such that is a natural transformation.

\[
\xymatrix{
M\ar[d]_{f}&F(M)\ar[r]^{i^{F}_{M}}&(\mathcal{P}(M);\subseteq)\\
N&F(N)\ar[u]^{F(f)}\ar[r]_{i^{F}_{N}}&(\mathcal{P}(N);\subseteq)\ar[u]_{f^{-1}}
}
\]
\end{Df}

\begin{Obs}
Let $(F,i)$ be a filter pair and $X$ be a set. The relation $\vdash\subseteq \mathcal{P}(Fm_{\Sigma}(X))\times Fm_{\Sigma}(X)$ such that for any $\Gamma\cup\{\varphi\}\subseteq Fm_{\Sigma}(X)$, $\Gamma\vdash \varphi$ iff for any $a\in F(Fm_{\Sigma}(X))$ if $\Gamma\subseteq i_{Fm_{\Sigma}(X)}(a)$ then $\varphi\in i_{Fm_{\Sigma}(X)}(a)$ is a tarskian consequence relation. Then we have a propositional logic associated with the filter pair $(F,i)$ such that the set of variables is $X$.

Below is the definition of a finitary filter pair so that its associated propositional logic is finitary.
\end{Obs}

\begin{Df}\label{DefinitionFilterPair}
Let $\Sigma$ be a signature. A {\bf finitary filter pair} over $\Sigma$ is a filter pair $(F,i)$ which $F$ is a functor from $\Sigma$-structures to algebraic lattices such that for any $M\in\Sigma\!-\!\text{str}$, $i_{M}$ preserves arbitrary infima (in particular $i_{M}(\top)=M$) and directed suprema.

\end{Df}






\begin{Df}[The category of filter pairs] \label{Ficat-df} 
Consider the category $\mathcal{F}i$ defined in the following manner:

\begin{itemize}
    \item \textbf{Objects:} Filters pairs $(F,i^{F})$.
    \item \textbf{Morphisms:} Let $(F, i^F)$ be a filter pair over a signature $\Sigma$ and $(F', i^{F'})$ be a filter pair over a signature $\Sigma'$. A morphism $(F,i^{F}) \to (F', i^{F'})$ is a pair $(H,j)$ such that $H:\Sigma'\!-\!str\to\Sigma\!-\!str$ is a {\em signature functor}
and $j : F'\Rightarrow F\circ H$ is a natural transformation such that given $M'\in Obj(\Sigma'\!-\!str)$, 
$$i^{F}_{H(M')}\circ j_{M'}=i^{F'}_{M'}.
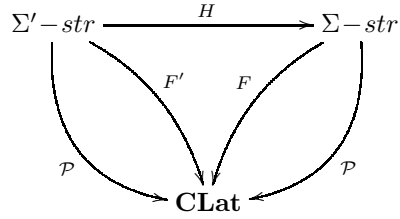
$$
\[
\xymatrix{
\Sigma'\!-\!str\ar[rr]^{H}\ar@/_2pc/[ddr]_{\mathcal{P}}\ar@/^1pc/[ddr]^{F'}&&\Sigma\!-\!str\ar@/^2pc/[ddl]^{\mathcal{P}}\ar@/_1pc/[ddl]_{F}\\
\\
&\mathbf{CLat}&
}
\]
    \item \textbf{Identities:} For each signature $\Sigma$ and each filter pair $(F, i^F)$ over $\Sigma$, $Id_{(F,i^F)} \coloneqq  (Id_{\Sigma\!-\!str}, Id_{F})$.
    \item \textbf{Composition:} Given morphisms $(H,j),(H',j')\, $ in $\, \mathcal{F}i$. $$(H',j')\bullet(H,j)=(H\circ H',j\bullet j')$$
          Where $(j\bullet j')_{M''}\coloneqq j_{H'(M'')}\circ j'_{M''}$.\\
          
          Observe that \[ i^{F}_{H\circ H'(M'')}\circ((j\bullet j')_{M''})=i^{F''}_{M''}\]
          
          Indeed:
          \begin{align*}
              i^{F}_{H\circ H'(M'')}\circ((j\bullet j')_{M''})&=i^{F}_{H\circ H'(M'')}\circ(j_{H'(M'')}\circ j'_{M''})\\
              &=(i^{F}_{H\circ H'(M'')}\circ j_{H'(M'')})\circ j'_{M''}\\
              &=i^{F'}_{H'(M'')}\circ j'_{M''}\\
              &=i^{F''}_{M''}
          \end{align*}

It is straightforward to check that the composition is associative and that identity laws hold.
\end{itemize}
\end{Df}

In \cite{AMP1}, a category of \emph{finitary} filter pairs was defined and regarded as another form of functorially encoding all finitary propositional logics: in fact, the category of propositional logics and flexible morphisms can be identified with a coreflective full subcategory of the category of filter pairs.

\begin{fact} \label{adjunct logic-filter}
\quad
\begin{itemize}

\item For any signature functor $H:\Sigma'-Str\to \Sigma-Str$, there is a signature morphism $m_{H}:\Sigma \to \Sigma'$, such that $m_{H}(c_{n})=\eta_{H}(X)(c_{n}(x_{0},...,x_{n-1}))$, {where $\eta_H(X) : Form_\Sigma(X) \to H(Form_{\Sigma`}(X))$} (see Lemma 3.17 of \cite{AMP1}). We consider the functor

\[\begin{array}{rccl}
\mathbb{L}:&\mathcal{F}i&\to&\Lf\\
&(G,i^{G})&&l_{G}\\
& \downarrow (H,j) &\mapsto&\downarrow m_{H}\\
&(G',i^{G'})&&l_{G'}
\end{array}
\]

\item The functor $\mathbb{F} : \cL \to \mathcal{F}i$

\[\begin{array}{rrcl}
\mathbb{F}:&\Lf&\to&\mathcal{F}i\\
&l&&(Fi_{l},\iota)\\
&h\downarrow&\mapsto&\mathbb{F}(h)\downarrow \\
&l'&&(Fi_{l'},\iota')
\end{array}
\]

where $\mathbb{F}(h)=(h^{\star},j^{\star})$  and the natural transformation 
 $j^{\star} : Fi_{l'}\Rightarrow Fi_{l}\circ h^{\star}$ is given by a family of inclusions, i.e., let $M'\in\Sigma'-str$ and $F'\in Fi_{l'}(M')$, then $j^{\star}_{M'}(F') := F'$.

\item The functor $\mathbb{F} : \cL \to \mathcal{F}i$ is full, faithful, injective on the objects and is left adjoint to the functor $\mathbb{L}$. By a well known result of category theory, the unity of this adjunction is an isomorphism. Moreover it is easy to see that the components of the natural transformation that is the unity of this adjunction is given, for each logic $l \in Obj(\Lf)$, by the identity $id_l : l \to \mathbb{L} \circ \mathbb{F}(l) = l$.

The components of the counit of this adjunction is given by, 
for each signature $\Sigma$ and each filter pair $(G,i^{G})$ over $\Sigma$:
$$(Id_{\Sigma-Str},j^{G}) :  (Fi_{l_G}, \iota) \to (G,i^{G})$$ 
where $j^G_M : G(M) \to \mathcal{F}i_{l_G}(M)$ is the unique factorization of $i^G_M : G(M) \to \wp(M)$ through $\iota_M : \mathcal{F}i_{l_G}(M) \hookrightarrow P(M)$. Thus for each logic $l'$,  $j^G$ induces by composition a (natural) bijection:

$$\mathcal{F}i(\mathbb{F}(l'), (G, i^G)) \ \cong \ 
\Lf(l', \mathbb{L}(G, i^G)).$$ 

\item The same constructions of the above functors provide a more general adjunction relating the category of filter pairs and propositional logics which are non-finitary.

\end{itemize}
\end{fact}

\begin{Prop}\label{fuctorfilterpairobject} Every filter pair $(F, i)$ over a signature $\Sigma$ determines an institution $I_{(F,i)}$ where:
\begin{itemize}
    \item $Sig_I = \Sigma\!-\!str$;
    \item $(Sig_I \xrightarrow{Sen_I} \set)$ = $(\Sigma\!-\!str \xrightarrow{forgetful} \set)$;
    \item $(Sig_I^{op} \xrightarrow{Mod_I} \CAT) = (\Sigma\!-\!str^{op} \xrightarrow{F} \mathbf{CLat} \rightarrowtail \CAT)$;
    \item for each $M \in Ob(Sig_I) = Ob(\Sigma\!-\!str)$, define $\models_M \subseteq Ob(Mod_I(M)) \times Sen_I(M) = F(M) \times |M|$ as:
    
    \[t \models_M m \quad if\!f \quad m \in i_M(t)\]
\end{itemize}
Moreover, when $i_M$ preserves arbitrary infima, the $\pi$-institution $P_{(F,i)}$ cannonically associated to $I_{(F,i)}$ is such that
for each $M \in Ob(Sig_I) = Ob(\Sigma\!-\!str)$, ${\cal C}_M : P(Sen_I) \to P(Sen_I)$ is given by 
$$(X \subseteq |M|) \ \mapsto \ i_M(t_X), $$ where $t_X := \bigwedge\{ t \in F(M): X \subseteq i_M(t)\}$ 

\end{Prop}

\Dem $Sig_{I},\ Sen_{I}$ and $Mod_{I}$ associated with a filter pair $(F,i)$ are well defined. It remains to prove the compatibility condition. Let $h:M\to M'$ be a morphism in $Sig_{I}=\Sigma\!-\!str$ and $a\in F(M')$ such that $a\models_{M'} h(m)$. So $h(m)\in i_{M'}(a)$ and since $i$ is a natural transformation we have $m\in h^{-1}\circ i_{M'}(a)=i_{M}\circ F(h)(a)$. Then $F(h)(a)\models_{M}m$. 

The associated $\pi$-instituion takes $X \sub P(U(M))$ into $i_M(T_X) = i_M (\bigwedge\{T \in F(M) : X \sub i_M(T)\} = \bigcap\{ i_M(T): X \sub i_M(T)\}$\\
\qed

\begin{Prop}{\bf (Every morphism of filter pair induces a institution morphism.)}\label{functorfilterpairmorphism} Given  morphism $(F,i) \xrightarrow{(H,j)} (F', i')$ then $I_{(F,i)}\xleftarrow{(H,Id,j)}I_{(F',i')}$ is a institution morphism.

\end{Prop} 
\Dem
We just need to prove that $(H,Id,j)$ satisifies the compatibility condition. Let $M'\in \Sigma'\!-\!str$, $m'\in F'(M')$ and $\varphi\in H(M')$.

\begin{align*}
    m'\models_{M'}Id_{M'}\varphi &\iff \varphi\in i'_{M'}(m')\\
    &\iff \varphi\in i_{H(M')}\circ j_{M'}(m')\\
    &\iff  j_{M'}(m')\models_{H(M')}\varphi\\
\end{align*}
The result follows\\
\qed

Using propositions \ref{fuctorfilterpairobject} and \ref{functorfilterpairmorphism} we can now define the (contravariant) functor:
\[\begin{tikzcd}
\mathcal{F}i \arrow[r, "D"]                      & \inm                                \\
{(F,i)} \arrow[r, maps to, shorten <=0.5mm, shorten >= 1.5mm] \arrow[d, "{(H,j)}"'{name=A}] & {I_{(F,i)}}\\ 
{(F',i')} \arrow[r, maps to, shorten <=0.5mm, shorten >= 1.5mm]                     & {I_{(F',i')}\arrow[u, "{(H,Id,j)}"{name=B}]}
\arrow[from=A,to=B,maps to, shorten <=0.5mm, shorten >= 1.5mm]
\end{tikzcd}\]
Verifying functoriality is straightforward.






\begin{Obs} 
    \begin{itemize}

\item From the adjunction  $\inm \rightleftarrows \pi-\inm$ described in section 2.1,  we obtain directly a functor $\mathcal{F}i \to \pi-\inm$.

\item From the adjuction $\Lf \rightleftarrows \mathcal{F}i$, recalled in Fact \ref{adjunct logic-filter}, we obtain functors $\Lf \to \inm$ and $\Lf \to \pi-\inm$.

\end{itemize}
\end{Obs}

\subsection{Generalizations}


In this final subsection we provide a kind of generalization of the previous subsections: we explore the extension of the category of propositional logics by the category of filter pairs to ``extend'' the ($\pi$-)institution of logics to a ($\pi$-)institution of filter pairs; we extend the concept of filter pairs allowing multialgebras as the domain of a filter pair and thus we extend the functor from filter pairs to the category of institutions to a funtor from the category of multifilter pairs to institutions.

\begin{Obs} 

The institution (respec. $\pi$-institution) associated to the abstract propositional logics as described in subsection 4.2 (respec. 4.1) can be ``extended'', through the adjunction $(\mathbb{F}, \mathbb{L}) : {\cal L}_f \rightleftarrows {\cal F}i$ (see Fact \ref{adjunct logic-filter}) to a institution (respec. $\pi$-institution) for the filter pairs (apart from size issues):

\begin{itemize}

\item 

* $\sig' = {\cal F}i$;

* $Sen' : \sig \to \mathbb{S}et$ is given by $((H,j) : (G,i^G) \to (G', i^{G'}))$ \ $\mapsto$ \ $(\eta_{H}(X) : Fm_\Sigma(X) \to H(Fm_{\Sigma'}(X)))$, where $G : \Sigma-Str^{op} \to$ {\bf CLat} and $G' : \Sigma'-Str^{op} \to$ {\bf CLat};

* For each $(G, i^G) \in |\sig|$ and $\Gamma \sub  Fm_\Sigma(X)$, we define $C'_{(G,i^G)}(\Gamma)\coloneqq  \{\phi \in  Fm_\Sigma(X) :\Gamma\vdash_{\mathbb{L}(G, i^G)} \phi\}$.

Denoting $(\mathbb{S}ig, Sen, (C_\bullet))$ the $\pi$-institution of propositional logics (subsection 4.1), note that:

*  $Sen'\circ \mathbb{F} = Sen$.

* For each $(\Sigma, \vdash) \in |\Lf|$,  $C'_{\mathbb{F}(\Sigma, \vdash)} = C_{(\Sigma, \vdash)} $.

Thus $(\mathbb{F}, id_{Sen})$ is, simultaneously, a morphism and a comorphism of $\pi$-institutions $(\mathbb{S}ig, Sen, (C_\bullet)) \to (\mathbb{S}ig', Sen', (C'_\bullet))$.

\item 

* $\sig' = {\cal F}i$;

* $Sen'_{}:\sig'_{}\to\set$ where $Sen'_{}(G,i^G)=\mathcal{P}(Fm_\Sigma(X))\times Fm_\Sigma(X)$ and given $(H,j)\in Mor_{\sig'}((G,i^G),(G',i^{G'}))$ then $Sen_{}(H,j):Sen_{}(G, i^G)\to Sen_{}(G',i^{G'})$ is such that $Sen_{}(H,j)(\langle\Gamma,\varphi\rangle)=\langle \eta_H(X)[\Gamma],\eta_H(X)(\varphi)\rangle$.
    
* $Mod'_{}:\sig'_{}\to\cat^{op}$ where $Mod'_{}(G,i^G) = Matr_{\mathbb{L}(G,i^G)}$ and given $(H,j)\in Mor_{\sig'}((G,i^G),(G',i^{G'}))$, $Mod'_{}(H,j) : Matr_{\mathbb{L}(G',i^{G'})}\to Matr_{\mathbb{L}(G,i^G)}$ such that $Mod'_{}(H,j)(\langle M',F'\rangle)=\langle H(M'),F'\rangle$. 

 * Given $(G,i^G))\in|\sig'_{}|$, $\langle M,F\rangle\in |Mod'_{}(G,i^G)|$ and $\langle\Gamma,\varphi\rangle\in Sen'_{}(G,i^G)$ define the relation $\models'_{(G,i^G)}\subseteq|Mod'_{}(G,i^G)|\times Sen'(G,i^G)$ as:
    
    \[\langle M,F\rangle\models'_{(G,i^G)}\langle\Gamma,\varphi\rangle\ if\!f\ for\ all\ \,\!Fm_\Sigma(X)\xrightarrow{v} M,\ v(\varphi)\in F \ for \ v[\Gamma]\subseteq F\]

Denoting $(\mathbb{S}ig, Sen, Mod, (\models_\bullet))$ the institution of propositional logics (subsection 4.2), note that:

*  $Sen'\circ \mathbb{F} = Sen$.

* $Mod'\circ \mathbb{F} = Mod$

* For each $l = (\Sigma, \vdash) \in |\Lf|$, each $ \langle \Gamma ,  \varphi\rangle \in Sen(l)$ and each $\langle M, F \rangle \in |Mod(l)|$ 
   \[\langle M,F\rangle\models'_{\mathbb{F}(l)}\langle\Gamma,\varphi\rangle\ if\!f\ \langle M,F\rangle\models_{l}\langle\Gamma,\varphi\rangle.\]

Thus $(\mathbb{F}, id_{Sen}, id_{Mod})$ is, simultaneously, a morphism and a comorphism of institutions $(\mathbb{S}ig, Sen, Mod, (\models_\bullet))$ $\to (\mathbb{S}ig', Sen', Mod', (\models'_\bullet))$.


    
    
\end{itemize}

\end{Obs}


{The institution obtained above can be extended to the case of multialgebras and that this also extends the institution for N-matrix semantics to propositional logic (\cite{AZ}) allowing us to use the institution theory in order to analyze logical properties of non-algebraizable logics. Moreover, another work in progress, we are trying, using filter pairs, to establish a multialgebraic semantics for propositional logics that are not algebraizable, for example Logic of Formal Inconsistency (LFI's) (\cite{CCM}), and possibly to obtain a kind of transfer theorem between metalogical and multialgebraic properties.
}








\begin{Obs}[Multialgebras]
\leavevmode

\begin{itemize}

    \item A $n$-ary multioperation on a set $A$ is a function $F : A^n \to {\cal P}^*(A)$, where ${\cal P}^*(A) = A \setminus \{ \emptyset\}$. To each ordinary $n$-ary multioperation on $A$, $f : A^n \to A$ is associated a (strict) $n$-ary operation on $A$ : $F : A^n \to {\cal P}^*(A)$ given by $F := s_A \circ f$, where $s_A : A \to {\cal P}^*(A), x \mapsto s_A(x) = \{x\}$.

    \item A multialgebraic signature is a sequence of pairwise disjoint sets 
 $\Sigma=(\Sigma_n)_{n\in \mathbb{N}},$
 where $\Sigma_n=S_n\sqcup M_n$, where $S_n$ is the set of strict multioperation symbols and $M_n$ is 
the set of multioperation symbols. In particular,  $\Sigma_0=S_0 \sqcup M_0$, $F_0$ 
is the set of symbols for constants and $M_0$ is the set of symbols for multiconstants. We also denote 
$\Sigma=((S_n)_{n\ge0},(M_n)_{n\ge0}).$

 

 

    \item A {multialgebra} over a signature $\Sigma=((S_n)_{n\ge0},(M_n)_{n\ge0})$, is a set $A$ endowed with a family of n-ary 
multioperations 
$$\sigma^A_n : A^n \to \mathcal P^*(A),\, \sigma_n \in S_n \sqcup M_n,\, n \in \mathbb N,$$  
such that: if $\sigma_n \in S_n$, then $\sigma^A_n : A^n \to \mathcal P^*(A)$ is a {\em strict} n-ary multioperation.

    \item If $A$ and $B$ are $\Sigma$-multialgebras, then a $\Sigma$-morphism from $A$ to $B$ is a function $h : A \to B$ such that  for each $n \in \mathbb{N}$, each $\sigma_n \in S_n \sqcup M_n$ and each $a_0, \cdots, a_{n-1} \in A$ 
$$h[\sigma^A(a_0, \cdots, a_{n-1})] \subseteq \sigma^B(h(a_0), \cdots, h(a_{n-1})).$$

\item $\Sigma$-morphisms between $\Sigma$-multialgebras can be composed in a natural way and they form a category $\Sigma$-Malg. It is clear that $\Sigma$-alg, the category of ordinary $\Sigma$-algebras is isomorphic to the a full subcategory of strict $\Sigma$-multialgebras. $s : \Sigma-Alg \hookrightarrow \Sigma-Malg$.


\item Every algebraic signature $\Sigma = (F_n)_{n \in \mathbb N}$ is  a multialgebraic signature where $M_n = \emptyset, \forall 
n \in \mathbb N$. Each algebra 
$(A, ((A^n \overset{f^A}\to A)_{f \in F_n})_{n \in \mathbb N})$
over the algebraic signature $\Sigma$ can be naturally identified with a multialgebra 
$(A, ((A^n \overset{f^A}\to A  \overset{s_A}\rightarrowtail \mathcal P^*(A))_{f \in F_n})_{n \in \mathbb N})$ 
over the same signature.

\item Every multialgebraic signature $\Sigma = ( (S_n)_{n \in \mathbb N}, (M_n)_{n \in \mathbb N})$ induces naturally a 
first-order language 
$L(\Sigma) = ((F_n)_{n \in \mathbb N}, (R_{n+1})_{n \in \mathbb N})$
where $F_n := S_n$ is the set of n-ary operation symbols and $R_{n+1} := M_n$ is the set of (n+1)-ary relation symbols. In this 
way, 
multialgebras 
$(A, ((A^n \overset{\sigma^A}\to  \mathcal P^*(A))_{\sigma  \in S_n \sqcup M_n})_{n \in \mathbb N})$
over a multialgebraic signature $\Sigma = (S_n \sqcup M_n)_{n \in \mathbb N}$ can be naturally identified with the first-order 
structures over the language  $L(\Sigma)$ that satisfies the $L(\Sigma)$-sentences:

$\forall x_0 \cdots \forall x_{n-1} \exists x_n (\sigma_n(x_0, \cdots, x_{n-1}, x_n)), \ \text{for each}\ \sigma_n \in R_{n+1} = 
M_n, n \in \mathbb N.$ \footnote{We will address this correspondence in Example \ref{MultialgebrasDLS}.}

\item Now we focus our attention into a more syntactic aspect of this multialgebras theory. We start with a (recursive) definition of 
(multi)terms: variables $x_i, i \in \mathbb N$ are terms; 
 if $t_0,\cdots,t_{n-1}$ are terms and $\sigma \in S_n \sqcup M_n$, then $\sigma(t_0,\cdots, t_{n-1})$ is a term.


\item To define an interpretation for terms, we need a preliminary step. Given $\sigma \in S_n \sqcup M_n$, we ``extend'' $\sigma^A 
: A^n \to \mathcal P^*(A)$ to a n-ary operation in $\mathcal P^*(A)$, $\sigma^{\mathcal P^*(A)} : \mathcal P^*(A)^n \to \mathcal 
P^*(A)$, by the rule:
$$\sigma^{\mathcal P^*(A)}(A_0, \cdots ,A_{n-1}) :=
\bigcup\limits_{a_0\in A_0} \cdots \bigcup\limits_{a_{n-1}\in A_{n-1}}\sigma^A(a_0,\cdots,a_{n-1}).$$
In this way, ${\cal P}^*(A)$ is an ordinary $\Sigma$-algebra. Moreover 
$$\sigma^{\mathcal P^*(A)}(\{a_0\}, \cdots ,\{a_{n-1}\}) = \sigma^A(a_0,\cdots,a_{n-1}).$$

\item The association above determines a functor $p : \Sigma-Malg \to \Sigma-alg$ and, the family of singleton maps $s_A : A \to (s\circ p)(A)$, $A \in |\Sigma-Malg|$, is a natural transformation.
\end{itemize}
\end{Obs}

\begin{Obs}[Multifilter pairs and institutions]
\leavevmode

\begin{itemize}

\item It is straightfoward to extend the notion of filter pair $(G, i^G)$, where the domain of the functor $G$ is the category $\Sigma-alg$ to the concept of {\em multifilter pair}, where the domain of the functor $G$ is the category $\Sigma-Malg$. With a natural   notion of morphism of mult-filter pair we obtain a category $m{\cal F}i$ of multifilter pairs.

\item The previously described functors $s : \Sigma-alg \to \Sigma-Malg$ and $p : \Sigma-Malg \to \Sigma-alg$ provide a pair of functors
$ {\cal F}i \rightleftarrows m{\cal F}i$.

\item The functor ${\cal F}i \to \inm$ can be extended to a funtor $m{\cal F}i \to \inm$.







\end{itemize}

\end{Obs}

 {We summarize below some of the functors previously presented.}



\[\xymatrix{
&\bf \pi-Inst_{mor}\ar@<0ex>[d]&\\
{\Lf}\ar@<0ex>[r]\ar@<0ex>[dr]&\bf Inst_{mor}\ar@<1ex>[u]&\\
&{\mathcal{F}i}\ar@<1ex>[ul]\ar@<0ex>[u]\ar@{->}[r]<1ex>& {m\mathcal{F}i}\ar@<0ex>[l] \ar@<0ex>[ul]
}\]

\section{Skolemization, a new institutional device}\label{Skolem}

Skolemization is an important tool of classical model theory, this section seeks to develop it in the context of institutions. We also prove a borrowing theorem and apply it to obtain a form of downward L\"owenheim-Skolem for the setting of multialgebras.

Given an institution $I$, we say that $\langle I, S, (\mathcal{I}_\Sigma)_{\Sigma\in|\sig|}, (\tau_\Sigma)_{\Sigma\in|\sig|} \rangle$ is an skolemization for $I$ iff:
\begin{itemize}
    \item $S$ is a functor of the form
    \[\begin{tikzcd}
(Mod)^\sharp \arrow[r, "S"]                                                      & (Mod^{\mathbb{P}res})^\sharp                                                       \\
{\langle \Sigma, M\rangle} \arrow[d, "{\langle f,u\rangle}"'{name=A}] \arrow[r, maps to] & {\langle (\Sigma_S,S_\Sigma),M_{S\Sigma}\rangle} \arrow[d, "{\langle g,v\rangle}"{name=B}] \\
{\langle \Sigma',N\rangle} \arrow[r, maps to]                                    & {\langle (\Sigma'_S,S_{\Sigma'}),N_{S\Sigma'}\rangle}                            \arrow[from=A,to=B,maps to, shorten <= 7mm, shorten >= 14mm]
\end{tikzcd}\]
    Where $^{\sharp}$ denotes the Grothendieck construction. We refer to S as the skolem functor.
    \item For each $\Sigma\in|\sig|$, $\Sigma\xrightarrow{\tau_\Sigma}\Sigma_S$ is an arrow in $\sig$ satisfying $M_{S\Sigma}\rest_{\tau_\Sigma}=M$ for all $M\in|Mod(\Sigma)|$. Given $M\in Mod(\Sigma)$ we say that $M'\in Mod(\Sigma_S)$ is a skolemization of $M$ if $M'\rest_{\tau_\Sigma}=M$ and $M'\models_{\Sigma_S}S_{\Sigma}$
    \item For each signature $\Sigma$, $\mathcal{I}_\Sigma$ is an inclusion system in $Mod(\Sigma_S)$ such that, if the $\Sigma_S$-models $M'$ and $N'$ are skolemizations of $M$ and $N$ respectively and $M'\hookrightarrow N'$ then $M^\star=N^\star$. \footnote{Given $M\in|Mod(\Sigma)|$, define $M^\star\coloneqq\{\varphi\in Sen(\Sigma): M\models_\Sigma \varphi\}$}
\end{itemize}

\begin{Ex}{\bf{FOL$^1$}}\\
    Let {\bf{\,FOL$^1$}} stand for the institution of unsorted first order logic and consider the functor: 
    \[\begin{tikzcd}
(Mod)^\sharp \arrow[r, "Skolem"]                                                      & (Mod^{\mathbb{P}res})^\sharp                                                       \\
{\langle \Sigma, M\rangle} \arrow[d, "{\langle f,u\rangle}"'{name=A}] \arrow[r, maps to] & {\langle (\Sigma_S,S_\Sigma),M_{S\Sigma}\rangle} \arrow[d, "{\langle f',u\rangle}"{name=B}] \\
{\langle \Sigma',N\rangle} \arrow[r, maps to]                                    & {\langle (\Sigma'_S,S_{\Sigma'}),N_{S\Sigma'}\rangle}                            \arrow[from=A,to=B,maps to, shorten <= 7mm, shorten >= 14mm]
\end{tikzcd}\]
Where $\Sigma_S$ and $S_\Sigma$ are, respectively, the skolem expansion and theory of $\Sigma$ and $M_{S\Sigma}$ is any skolemization of $M$ with the same underlying set. Let $F^\Sigma_\psi$ be the skolem function of the $\Sigma$-formula $\psi$ and define $f'$ as follows: if $x\in\Sigma$ simply let $f'(x)=f(x)$, else we have $x=F^{\Sigma}_\psi$ for some $\psi$ in $Sen(\Sigma)$ and then we let $f'(x)=F^{\Sigma'}_{Sen\,f(\psi)}$.

For each first order signature $\Sigma$, let $\mathcal{I}_\Sigma$ be the usual inclusion system on $Mod^{{\bf{FOL^1}}}(\Sigma)$ and define $\tau_\Sigma:\Sigma\to\Sigma_S$ as $\tau_\Sigma(x)=x$. It is easy to see that $$\langle {\bf{FOL^1}}, Skolem, (\mathcal{I}_\Sigma)_{\Sigma\in|\sig^{{\bf{FOL^1}}}|}, (\tau_\Sigma)_{\Sigma \in|\sig^{{\bf{FOL^1}}}|}\rangle$$ is a skolemization for {\bf{FOL$^1$}}.

\end{Ex}
\begin{theorem}
    Let I institution with skolemization $\langle I, S, (\mathcal{I}_\Sigma)_{\Sigma\in|\sig^I|}, (\tau_\Sigma)_{\Sigma\in|\sig^I|} \rangle$. Given an institution $J$ and a morphism $\langle\phi,\alpha,\beta \rangle:J\to\ I$ if:
    \begin{itemize}
        \item $\phi$ is fully faithful,
        \item For each $\Sigma_i\in|\sig^I|$ there is some $\Sigma_j\in|\sig^J|$ such that $\phi(\Sigma_j)\cong(\phi\Sigma_i)_S$ in $\sig^I$. Let $i_{\Sigma_i}:(\Sigma_j)\to(\Sigma_i)_S$ denote the isomorphism arrow,
        \item Each $\beta_\Sigma$ is an isomorphism, and
        \item Each $\alpha_\Sigma$ is semantically surjective, that is, for every $\varphi\in Sen^{J}(\Sigma)$ there is some $\psi\in \alpha_{\Sigma}[Sen^I(\phi\Sigma)]$ such that $\varphi^\star=\psi^\star$.
    \end{itemize}
    Then $\langle J, S', (\mathcal{I'}_\Sigma)_{\Sigma\in|\sig^J|}, (\tau'_\Sigma)_{\Sigma\in|\sig^J|} \rangle$ has a skolemization where
    \begin{itemize}
        \item If $\mathcal{I}_{\phi\Sigma}=\langle I,E\rangle$ then $\mathcal{I'}_{\Sigma}=\langle I',E' \rangle$ where $I'$ and $E'$ are the images of $\beta^{-1}_{\widecheck{\Sigma}}Mod^Ii_{\phi\Sigma}$ restricted to $I$ and $E$ respectively,
         \item For each $\Sigma$, $\tau'_\Sigma$ is the unique arrow satisfying $\phi(\tau'_\Sigma)=i^{-1}_{\phi\Sigma}\cdot\tau_\phi\Sigma$.
    \end{itemize}
\end{theorem}
\Dem
Consider the application

\[\begin{tikzcd}
m:(Mod^J)^\sharp \arrow[rr]                                                        &  & (Mod^I\phi)^\sharp                                                                                   \\
{\langle \Sigma, M \rangle} \arrow[rr, maps to,xshift=0.5ex,shorten=5mm] \arrow[d, "{\langle f,u \rangle}"{name=U}] &  & {\langle \phi(\Sigma),\beta_\Sigma(M)\rangle} \arrow[d, "{\langle \phi(f),\beta_\Sigma(u)\rangle}"{name=D}] \\
{\langle \Sigma',N\rangle} \arrow[rr, maps to,xshift=0.3ex,shorten=4.8mm]                                     &  & {\langle \phi(\Sigma'),\beta_{\Sigma'}(N)\rangle}
\arrow[from=U,to=D,shorten=11mm, xshift=-4.2ex,maps to]
\end{tikzcd}
\]

Let us prove that $m$ is a functor. Given arrows $\langle\Sigma,M\rangle\xrightarrow{\langle f,u\rangle}\langle \Sigma',N\rangle\xrightarrow{\langle g,v \rangle}\langle \Sigma'',W \rangle$ in $(Mod^J)^\sharp$ we have:

\begin{align*}
    m(\langle g,v \rangle \cdot \langle f,u\rangle) &= m(\langle gf, Mod^Jfv\cdot u\rangle)\\
                                                    &= \langle \phi(gf),\beta_\Sigma(Mod^Jfv\cdot u)\rangle\\
                                                    &= \langle \phi(g)\cdot\phi(f),(\beta_\Sigma Mod^Jf)(v)\cdot \beta_\Sigma (u) \rangle\\
                                                    &= \langle \phi(g)\cdot\phi(f),(Mod^I\phi(f)\beta_{\Sigma'})(v)\cdot \beta_\Sigma (u) \rangle\\
m(\langle g,v\rangle)\cdot m(\langle f,u \rangle)    &= \langle \phi(g),\beta_{\Sigma'}(v)\rangle\cdot \langle \phi(f),\beta_\Sigma(u)\rangle
\end{align*}
As $m$ clearly satisfies the identity laws we have that $m$ is well defined.

Consider now the functors $(Mod^I\phi)^\sharp \xhookrightarrow{\mathcal{J}} (Mod^I)^\sharp \xrightarrow{S} (Mod^{\pres^I})^\sharp$. Composing:

\[\begin{tikzcd}
(Mod^J)^\sharp \arrow[rr, "{S\mathcal{J}m}"]                                                        &  & (Mod^{\pres^I})^\sharp                                                                                   \\
{\langle \Sigma, M \rangle} \arrow[rr, maps to,xshift=0.3ex,shorten=5mm] \arrow[d, "{\langle f,u \rangle}"{name=U}] &  & {\langle ((\phi\Sigma)_S,S_{\phi\Sigma}),(\beta_\Sigma(M))_{S\phi\Sigma}\rangle} \arrow[d, "{\langle \psi,v\rangle}"{name=D}] \\
{\langle \Sigma',N\rangle} \arrow[rr, maps to,xshift=0.3ex,shorten <=4.8mm, shorten >=4mm]                                     &  & {\langle ((\phi\Sigma')_S,S_{\phi\Sigma'}),(\beta_{\Sigma'}(N))_{S_{\phi\Sigma}}\rangle}
\arrow[from=U,to=D,shorten=16mm, xshift=-7.6ex,maps to]
\end{tikzcd}
\]

We now have what we need to define a functor
$S':(Mod^J)^\sharp\to(Mod^{\pres^J})^\sharp$. Given $\langle\Sigma,M\rangle\in|(Mod^J)^\sharp|$, let $S'(\langle\Sigma,M\rangle)\coloneqq\langle (\widecheck{\Sigma},S_{\widecheck{\Sigma}}),M_{\widecheck{\Sigma}}\rangle$ where:

\begin{itemize}
    \item $\widecheck{\Sigma}$ is an object in $\sig^J$ such that there is an isomorphism $i_{\phi\Sigma}:\phi(\widecheck{\Sigma})\xrightarrow{\sim}(\phi\Sigma)_S$ in $\sig^I$
    \item $S_{\widecheck{\Sigma}}\coloneqq\alpha_{\widecheck{\Sigma}}(Sen^Ii^{-1}_{\phi\Sigma}(S_{\phi\Sigma}))$
    \item $\widecheck{M}\coloneqq\beta_{\widecheck{\Sigma}}^{-1}Mod^Ii_{\phi\Sigma}((\beta_\Sigma)_{S\phi\Sigma})$
 \end{itemize}
And, given an arrow $\langle f,u\rangle$ in $(Mod^J)^\sharp$, let $S'(\langle f,u\rangle)\coloneqq\langle\widecheck{\psi},\widecheck{v} \rangle$, where:

\begin{itemize}
    \item $\phi(\widecheck{\psi})$ is the lone arrow that makes the below square commute
\[\xymatrix{
(\phi\Sigma)_S\ar[r]^{\psi}\ar[d]_{\cong}             &(\phi\Sigma')_S\ar[d]^{\cong}\\
\phi(\widecheck{\Sigma})\ar@{-->}[r]_{\phi(\widecheck{\psi})} &\phi(\widecheck{\Sigma'})
}\]
    \item $\widecheck{v}\coloneqq\beta_{\widecheck{\Sigma}}^{-1}(Mod^Ii_{\phi\Sigma}(v))$
\end{itemize}

First, let us prove that $S'(\langle f,u\rangle)$ is a morphism in $(Mod^{\pres^J})^\sharp.$

\begin{align*}
    Sen^I\psi(S_{\phi\Sigma}) &\subseteq S_{\phi\Sigma'}\\
    \alpha_{\widecheck{\Sigma'}}(Sen^Ii^{-1}_{\phi\Sigma'}(Sen^I\psi(S_{\phi\Sigma})) &\subseteq \alpha_{\widecheck{\Sigma'}}(Sen^Ii^{-1}_{\phi\Sigma'}(S_{\phi\Sigma'}))
\end{align*}

As $\alpha_{\widecheck{\Sigma'}}\cdot Sen^Ii^{-1}_{\phi\Sigma'}\cdot Sen^I\psi=\alpha_{\widecheck{\Sigma'}}\cdot Sen^I\phi\widecheck{\psi}\cdot Sen^Ii^{-1}_{\phi\Sigma}=Sen^J\widecheck{\psi}\cdot\alpha_{\widecheck{\Sigma}}\cdot Sen^Ii^{-1}_{\phi\Sigma}$ it follows that $Sen^J\widecheck{\psi}(S_{\widecheck{\Sigma}})\subseteq S_{\widecheck{\Sigma'}}$.

Now, we prove that $S'$ is functorial. It is clear that $S'\langle1_\Sigma,1_M\rangle=\langle 1_{\widecheck{\Sigma}}, 1_{\widecheck{M}}\rangle=1_{S'\langle\Sigma,M\rangle}$ and, given a pair of arrows  
\[\langle((\phi\Sigma)_S,S_{\phi\Sigma}),(\beta_\Sigma M)_{S\phi\Sigma}\rangle\xrightarrow{\langle \psi_1,w\rangle}\langle((\phi(\Sigma'))_S,S_{\phi(\Sigma')}),(\beta_\Sigma' N)_{S\phi(\Sigma')}\rangle\]
\[and\]
\[\langle((\phi(\Sigma'))_S,S_{\phi(\Sigma')}),(\beta_\Sigma' N)_{S\phi(\Sigma')}\rangle\xrightarrow{\langle \psi_2,y\rangle}\langle((\phi(\Sigma''))_S,S_{\phi(\Sigma'')}),(\beta_\Sigma'' W)_{S\phi(\Sigma'')}\rangle\]
We have:
\[\begin{tikzcd}
(\phi\Sigma)_S \arrow[d, "\cong"'] \arrow[r, "\psi_1"]          & (\phi\Sigma')_S \arrow[r, "\psi_2"] \arrow[d, "\cong" description] & (\phi\Sigma'')_S \arrow[d, "\cong"] \\
\phi(\widecheck{\Sigma}) \arrow[r, "\phi(\widecheck{\psi_1})"', dashed] & \phi(\widecheck{\Sigma'}) \arrow[r, "\phi(\widecheck{\psi_2})"', dashed]   & \phi(\widecheck{\Sigma''})             
\end{tikzcd}\]

Notice that, by definition, $\phi(\widecheck{\psi_2\cdot\psi_1})$ is the unique arrow that makes the outer rectangle commute. It follows that 
$\phi(\widecheck{\psi_2\cdot\psi_1})=\phi(\widecheck{\psi_2})\cdot\phi(\widecheck{\psi_1})$ and so, by faithfulness,  $\widecheck{\psi_2\cdot\psi_1} = \widecheck{\psi_2}\cdot\widecheck{\psi_1}$.

Moreover, let $\bullet$ and $\circ$ stand for the composition of the second coordinate in, respectively, $(Mod^J)^\sharp$ and $(Mod^{\pres^J})^\sharp$. We then have:
\begin{align*}
        \widecheck{w}\circ\widecheck{y} &= Mod^J\widecheck{\psi_{1}}\beta^{-1}_{\widecheck{\Sigma'}}Mod^Ii_{\phi\Sigma}(w)\cdot\beta_{\widecheck{\Sigma}}^{-1}Mod^Ii_{\phi\Sigma}(y)\\
    &= \beta^{-1}_{\widecheck{\Sigma}}Mod^I\phi\widecheck{\psi_{1}}Mod^Ii_{\phi\Sigma'}(w)\cdot  \beta^{-1}_{\widecheck{\Sigma}}Mod^Ii_{\phi\Sigma}(y)\\
    &= \beta^{-1}_{\widecheck{\Sigma}} Mod^Ii_{\phi\Sigma}Mod^J\psi_1(w)\cdot \beta^{-1}_{\widecheck{\Sigma}} Mod^Ii_{\phi\Sigma}(y)\\
\widecheck{w\bullet y}    &= \beta^{-1}_{\widecheck{\Sigma}} Mod^Ii_{\phi\Sigma}(Mod^J\psi_1(w)\cdot y)
\end{align*}

We now have a functor $S':(Mod^J)^\sharp\to(Mod^{\pres^J})^\sharp$. Finally, let us prove that $S'$ indeed forms a skolemization.

First, notice that $i^{-1}_{\phi\Sigma}\cdot\tau_{\phi\Sigma}\in Sig^I(\phi\Sigma,\phi\widecheck{\Sigma})$. Define then ${\tau'_{\Sigma}}$ as the arrow in $Sig^J(\Sigma,\widecheck{\Sigma})$ satisfying $\phi(\widecheck{\tau})=i^{-1}_{\phi\Sigma}\cdot\tau$. Given some $M\in|Mod^J\Sigma|$ we have:

\begin{align*}
    \widecheck{M}\rest_{\widecheck{\tau}} &= Mod^J\widecheck{\tau}\cdot\beta^{-1}_{\widecheck{\Sigma}}Mod^Ii_{\phi\Sigma}((\beta_\Sigma(M))_{S\phi\Sigma})\\
    &= \beta^{-1}_\Sigma(Mod^I\phi\widecheck{\tau}Mod^Ii_{\phi\Sigma}((\beta_\Sigma(M))_{S\phi\Sigma}))\\
    &= \beta^{-1}_\Sigma(Mod^I\tau((\beta_\Sigma(M))_{S\phi\Sigma})\\
M   &= \beta^{-1}_\Sigma(\beta_\Sigma(M))
\end{align*}

Now given $\mathcal{I}_{\phi\Sigma}=\langle\mathcal{U},E \rangle$ we define $\mathcal{I'}_{\Sigma}=\langle \mathcal{U'},E' \rangle$ as:

\begin{itemize}
    \item For any object $i$ in $\mathcal{U}$, $\beta^{-1}_{\widecheck{\Sigma}}Mod^Ii_{\phi\Sigma}(i)$ is an object of $\mathcal{U'}$\\
          For any arrow $ a$ in $\mathcal{U}$, $\beta^{-1}_{\widecheck{\Sigma}}Mod^Ii_{\phi\Sigma}(a)$ is an arrow of $\mathcal{U'}$
    \item For any object $e$ in ${E}$, $\beta^{-1}_{\widecheck{\Sigma}}Mod^Ii_{\phi\Sigma}(e)$ is an object of ${E'}$\\
          For any arrow $ b$ in ${E}$, $\beta^{-1}_{\widecheck{\Sigma}}Mod^Ii_{\phi\Sigma}(b)$ is an arrow of ${E'}$
\end{itemize}

Routine calculations show $\mathcal{I'}_{\Sigma}$ is an inclusion system in $Mod^J\widecheck{\Sigma}$.

Finally, suppose that the $\widecheck{\Sigma}$-models $M'$ and $N'$ are skolemizations of, respectively, the $\Sigma$-models $M$ and $N$ and that $M'\xhookrightarrow{}N'$. Clearly then $(\beta_{\widecheck{\Sigma}}(M'))\rest_{i^{-1}_{\phi\Sigma}}\xhookrightarrow{}(\beta_{\widecheck{\Sigma}}(N'))\rest_{i^{-1}_{\phi\Sigma}}$.
Moreover, using structurality and the morphism compatibility condition we have that:

\[M'\models_{\widecheck{\Sigma}}S_{\widecheck{\Sigma}} \iff M'\models \alpha_{\widecheck{\Sigma}}(Sen^Ii^{-1}_{\phi\Sigma}(S_{\phi\Sigma})) \iff Mod^Ii^{-1}_{\phi\Sigma}\beta_{\widecheck{\Sigma}}(M')\models_{(\phi\Sigma)_S}S_{\phi\Sigma}\]

It follows then that 

\[((\beta_{\widecheck{\Sigma}}(M'))\rest_{i^{-1}_{\phi\Sigma}\cdot\tau})^{\star}=((\beta_{\widecheck{\Sigma}}(N'))\rest_{i^{-1}_{\phi\Sigma}\cdot\tau})^{\star}\]

Or equivalently,

\[((\beta_{\widecheck{\Sigma}}(M'))\rest_{i^{-1}_{\phi\widecheck{\tau}}})^{\star}=((\beta_{\widecheck{\Sigma}}(N'))\rest_{i^{-1}_{\phi\widecheck{\tau}}})^{\star}\]

By naturality,

\[(\beta_{\Sigma}(Mod^{I}\widecheck{\tau}(M')))^{\star}=(\beta_{\Sigma}(Mod^{I}\widecheck{\tau}(N')))^{\star}\]

Since $M'$ and $N'$ are skolemizations, we have that $M'\rest_{\widecheck{\tau}}=M$ and $N'\rest_{\widecheck{\tau}}=N$. Now notice that

\[M\models \alpha_\Sigma(\varphi) \iff \beta_\Sigma(M)\models \varphi \iff \beta_\Sigma(N)\models \varphi \iff N\models\alpha_\Sigma(\varphi) 
\]

As $\alpha_\Sigma$ is semantically surjective the result follows.\\
\qed

As an illustration of the previous theorem we present the following:

\begin{Ex}\label{MultialgebrasDLS}{\textbf{(Multialgebras have the Downward L\"owenheim-Skolem property)}}\\
     We now describe {\bf{MA}}---the institution of (unsorted) multialgebras\footnote{Here we consider a wide sense of  $n$-ary multioperation on a set $A$: this is just a function $F : A^n \to {\cal P}(A)$, allowing  $\emptyset$ in the range.}. As signatures we simply use (unsorted) first order signatures. The intuition here is that function symbols are to be interpreted as functions and relations as multioperations.\\
     Let us describe the syntax. The terms are built in a first order manner with the caveat that relation symbols can too be used to form terms, that is, functions are allowed to take relations as arguments and we can compose relations. For the formulas, we have two atoms: $t\succ t'$, interpreted as set inclusion, and $t\doteq t'$, interpreted as (deterministic) equality. The full set of formulas is built by using quantification and Boolean connectives, the sentences being the formulas without free variables. For the semantics we let the category of models of given signature be the category of multialgebras of that signature. A more detailed characterization of this institution can be found in \cite{Lamo}.\\
     We can now describe a morphism {\bf{MA}}$\xrightarrow{\langle \phi,\alpha,\beta \rangle}$ {\bf{FOL$^1$}}:
     \begin{itemize}
         \item We start by defining the functor
         \[\begin{tikzcd}
\phi:\sig^{{\bf{MA}}} \arrow[r]                                                  & \sig^{{\bf{FOL^1}}}                                              \\
{\langle(\mathcal{F}_i)_{i<\omega},(\mathcal{M}_i)_{i<\omega} \rangle} \arrow[d, "f"'{name=A}] \arrow[r, maps to, shorten <= 1.2mm,shorten >= 1.3mm] & {\langle (\mathcal{F}_i)_{i<\omega},(\mathcal{R}}_i)_{i<\omega}\rangle \arrow[d, "f"{name=B}] \\
{\langle (\mathcal{F'}_i)_{i<\omega},(\mathcal{M'}_i)_{i<\omega}\rangle} \arrow[r, maps to,shorten >= 0.5mm]               & {\langle (\mathcal{F'}_i)_{i<\omega},(\mathcal{R'}_i)_{i<\omega}\rangle}             \arrow[from=A,to=B,maps to,shorten <=17.5mm, shorten >= 17mm]
\end{tikzcd}\]
Where $\mathcal{R}_{i+1}\coloneqq\{r_m:m\in M_i\}$. It is easy to see that $\phi$ is well defined and fully faithful. Moreover, we have that the functor is essentially surjective.
         \item  Given $\Sigma\in|\sig^{{\bf{MA}}}|$ we define $\alpha_{\Sigma}:Sen^{{\bf{FOL^1}}}(\phi\Sigma)\to Sen^{{\bf{MA}}}(\Sigma)$ recursively:
         \begin{align*}
             \alpha_{\Sigma}(x_i)                  &= x_i\\
             \alpha_{\Sigma}(f(t_1\cdots t_n))     &= f(\alpha_{\Sigma}(t_1)\cdots\alpha_{\Sigma}(t_n))\\
             \alpha_{\Sigma}(t\approx t')          &= \alpha_{\Sigma}(t)\doteq \alpha_{\Sigma}(t')\\
             \alpha_{\Sigma}(r_m(t_1\cdots t_{n+1})) &= m(\alpha_{\Sigma}(t_1)\cdots \alpha_{\Sigma}(t_n))\succ t_{n+1}\\
         \end{align*}
         \vspace{-1cm}\[             \alpha(A\land B) = \alpha_{\Sigma}(A)\land\alpha_{\Sigma}(B);\quad\!\! \alpha_{\Sigma}(\neg A) = \neg\alpha_{\Sigma}(A);\quad\!\! \alpha_{\Sigma}(\exists x_i(A)) =\exists x_i (\alpha_{\Sigma}(A))\]
         Elementary induction shows that $\alpha$ is indeed a natural transformation.\\
         Notice that the set $\alpha_{\Sigma}[Sen^{{\bf{FOL^1}}}(\phi\Sigma)]$ consists of formulas built of terms where there is no composition with multioperations. The idea we use to show that $\alpha_\Sigma$ is semantically surjective is simple: suppose we have the formula $f(x_1\cdots m(y_1\cdots y_k)\cdots x_n)\doteq x_{n+1}$ where $m(y_1\cdots y_k)$ happens in the $j$-th place, we simply introduce a new variable and restrict its domain, i.e., we consider the formula  $\forall x_j(m(y_1\cdots y_k)\succ x_j\land f(x_1\cdots x_j\cdots x_n)) \doteq x_{n+1}$. Using a similar technique for inclusion\footnote{For example, if $f$ and $g$ are function symbols and $m$ is a multioperation, then the formula $f(m(x))\succ g(y)$ is equivalent to $\exists z ((m(x)\succ z )\land (f(z) \doteq g(y))) $ } and proceeding by induction on nested formulas the proof follows.\footnote{Note that the full proof would have to address equalities between multioperations and inclusions between functions. The former being equivalent to $\bot$ and the latter to an equality, for instance, $f(x)\succ g(y)$ and $f(x)\doteq g(y)$}
         \item Given some signature $\Sigma$ consider the functor
         \[\begin{tikzcd}
\beta_\Sigma:Mod^{{\bf{MA}}}(\Sigma) \arrow[r]                   & Mod^{{\bf{FOL^1}}}(\phi\Sigma)                 \\
{\langle W,(F_i)_{i<\omega},(M_i)_{i<\omega}\rangle} \arrow[r, maps to, shorten <= 0.9mm,shorten >= 0.7mm] \arrow[d, "h"'{name=A}] & {\langle W,(F_i)_{i<\omega},(R_i)_{i<\omega}\rangle} \arrow[d, "h"{name=B}] \\
{\langle W',(F'_i)_{i<\omega},(M'_i)_{i<\omega}\rangle} \arrow[r, maps to]              & {\langle W',(F'_i)_{i<\omega},(R'_i)_{i<\omega}\rangle}  
\arrow[from=A,to=B,mapsto, shorten >= 18.7mm, shorten <= 19.1mm]
\end{tikzcd}\]
        Where $r_m =\{x_1x_2\cdots x_{i+1}\in M^{i+1}:x_{i+1}\in m(x_1\cdots x_i)\} $ and ${R}_{i+1}\coloneqq\bigcup_{m\in M_i}r_m$. It is easy to see that $\beta_\Sigma$ is well defined and that $(\beta_\Sigma)_{\Sigma\in|\sig^{{\bf{MA}}}|}$ ensemble into a natural transformation. Furthermore simple arguments show that $\langle \phi,\alpha, \beta \rangle$ indeed forms an institution morphism.\\
        Finally, we define an inverse for $\beta_\Sigma$
        \[\begin{tikzcd}
Mod^{{\bf{MA}}}(\Sigma)                  & Mod^{{\bf{FOL^1}}}(\phi\Sigma):\beta_\Sigma^{-1} \arrow[l]                  \\
{\langle W,(F_i)_{i<\omega},(M_i)_{i<\omega}\rangle} \arrow[d, "h"'{name=A}] & {\langle W,(F_i)_{i<\omega},(R_i)_{i<\omega}\rangle} \arrow[d, "h"{name=B}] \arrow[l, maps to, shorten <= 0.9mm,shorten >= 0.7mm]\\
{\langle W',(F'_i)_{i<\omega},(M'_i)_{i<\omega}\rangle}             & {\langle W',(F'_i)_{i<\omega},(R'_i)_{i<\omega}\rangle} \arrow[l, maps to]  
\arrow[from=B,to=A,mapsto, shorten >= 18.7mm, shorten <= 19.1mm]
\end{tikzcd}\]
    Where $m_r(x_1\cdots x_i):=\{x_{i+1}\in W:r(x_1\cdots x_{i+1})\}$ and $M_i:=\bigcup_{r\in R_{i+1}}m_r$.
     \end{itemize}    
     This proves that {\bf{MA}} has a skolemization. Observe that the inclusion system of this skolemization is the standard one, that is, an inclusion simply means a subalgebra. Using this fact and a similar technique to skolem hulls one can now easily prove a downward L\"owenheim-Skolem result for multialgebras.
\end{Ex}



\section{Final remarks and future works}

We finish the present work presenting some  perspectives of future developments.



\begin{Obs}
The adjunctions obtained in Section 2 lead us to research about the relationship between the types of representations of propositional logics and their institutions and $\pi$-institution developed in Section 4:

\begin{enumerate}
    \item The result of these analyzes may provide us with a way to study metalogical properties of abstract propositional logics and their algebraic or categorical properties, for instance, the relation between Craig's interpolation in an abstract logics and the amalgamation properties of its algebraic or categorical semantic. In particular, it could be interesting examine the possibility of generalize the work in \cite{AMP2}, describing a Craig interpolation property for institutions associated to multialgebras: this is a natural (non-deterministic) matrix semantics for complex logics as the {\bf LFI's}, the logics of formal inconsistencies (see \cite{CFG}).
    
    
    \item By a convenient modification of this matrix institution, is presented in section 3.2 of \cite{MaPi3} an institution for each ``equivalence class" of algebraizable logic: this furnished technical means  to  apply notions and results from the theory of institutions in the propositional logic setting  and to derive, from the introduction of the notion of ``Glivenko's context", a strong and general form of Glivenko's Theorem relating two ``well-behaved" logics.

\end{enumerate}

\end{Obs}





\begin{Obs}
Another interesting discussion $-$ already suggested in \cite{Diac2} $-$ which can be posed is how to repeat the whole discussion of Section 3 with a version of the Grothendieck construction for indexed $2$-categories in order to directly produce the $2$-category of institutions, as well as related $2$-categories of institution-like structures. The technical categorical devices necessary for developing this idea are presented in \cite{Bak2}, for example.
\end{Obs}

\begin{Obs}
    The borowing result presented in section \ref{Skolem} leads us to question which institutions have the skolemization property in a non-trivial way. Furthermore, in predicate logic skolemization is deeply related to the idea of indiscernibles, which leads the authors to question if an institution-independent formalization of this idea is possible. Another question is if whether skolemization of an institution $I$ implies the skolemization of $\pres^I$; if so, then in any skolemizable institution every theory would admit some expansion to a model-complete theory. 
\end{Obs}


\end{document}